\documentclass[a4paper,10pt]{scrartcl}
\usepackage[utf8x]{inputenc}
\usepackage[T1]{fontenc}
\usepackage[ngerman,english]{babel}
\usepackage{amsmath, amsthm, amssymb} 
\usepackage[usenames,dvipsnames]{color}
\usepackage{tikz}
\allowdisplaybreaks

\newcommand{\eps}{\varepsilon}
\newcommand{\N}{\mathbb{N}}
\newcommand{\R}{\mathbb{R}}
\newcommand{\Rn}{\R^n}

\newcommand{\nat}{\in\N}

\newcommand{\one}{1\!\!1}

\newcommand{\ny}{\nu}
\newcommand{\my}{\mu}
\newcommand{\phii}{\varphi}

\newcommand{\col}{\colon}
\newcommand{\Om}{\Omega}
\newcommand{\na}{\nabla}

\newcommand{\Laplace}{\Delta}
\newcommand{\Lap}{\Laplace}

\newcommand{\ctilde}{\widetilde{c}}

\newcommand{\ttilde}{\widetilde{t}}
\newcommand{\utilde}{\widetilde{u}}
\newcommand{\vtilde}{\widetilde{v}}

\newcommand{\Ctilde}{\widetilde{C}}

\newcommand{\Mtilde}{\widetilde{M}}

\newcommand{\Ttilde}{\widetilde{T}}

\newcommand{\epstilde}{{\widetilde{\eps}}}

\newcommand{\mhat}{\widehat{m}}

\newcommand{\uhat}{\widehat{u}}
\newcommand{\vhat}{\widehat{v}}

\newcommand{\Chat}{\widehat{C}}

\newcommand{\Qbar}{\overline{Q}}

\newcommand{\rand}{\del\Omega}
\newcommand{\Ombar}{\overline{\Om}}

\newcommand{\amrand}{|_{\rand}}

\newcommand{\vonbis}[2]{|_{#1}^{#2}}

\newcommand{\del}{\partial}

\newcommand{\delny}{\partial_\ny}
\newcommand{\intom}{\int_\Om}
\newcommand{\intnr}{\int_0^r}
\newcommand{\intnT}{\int_0^T}
\newcommand{\intnt}{\int_0^t}

\newcommand{\inttaut}{\int_\tau^t}
\newcommand{\intntn}{\int_0^{t_0}}

\newcommand{\intnR}{\int_0^R}
\newcommand{\intdlR}{\int_\delta^R}
\newcommand{\intnh}{\int_0^h}
\newcommand{\inttth}{\int_t^{t+h}}
\newcommand{\inthtnh}{\int_h^{t_0+h}}
\newcommand{\inttntnh}{\int_{t_0}^{t_0+h}}
\newcommand{\inttntnd}{\int_{t_0}^{t_0+\delta}}

\newcommand{\intn}[1]{\int_0^{#1}}

\newcommand{\Tmax}{T_{max}}

\newcommand{\Liom}{L^\infty(\Om)}
\newcommand{\Lqom}{L^q(\Om)}
\newcommand{\Lpom}{L^p(\Om)}
\newcommand{\Lzom}{L^2(\Om)}
\newcommand{\Leom}{L^1(\Om)}
\newcommand{\LeQT}{L^1(Q_T)}

\newcommand{\Li}{L^\infty}

\newcommand{\Weqom}{W^{1,q}(\Om)}
\newcommand{\Wzqom}{W^{2,q}(\Om)}
\newcommand{\Wezom}{W^{1,2}(\Om)}
\newcommand{\WeiR}{W^{1,\infty}(\R)}

\newcommand{\Weeom}{W^{1,1}(\Om)}

\newcommand{\LiomnTn}{L^\infty(\Om\times(0,T_0))}
\newcommand{\Liomntn}{L^\infty(\Om\times(0,t_0))}
\newcommand{\Liomnt}{L^\infty(\Om\times(0,t))}
\newcommand{\Liomtaut}{L^\infty(\Om\times(\tau,t))}
\newcommand{\LilocNTmaxWeqom}{L^\infty_{loc}([0,\Tmax);\Weqom)}

\newcommand{\LilocnTWeqom}{L^\infty_{loc}((0,T);\Weqom)}
\newcommand{\LiNTX}[1]{L^\infty([0,T];#1)}
\newcommand{\LpnTX}[1]{L^p((0,T);#1)}
\newcommand{\LiQT}{L^\infty(Q_T)}

\newcommand{\LeomnT}{L^1(\Om\times(0,T))}

\newcommand{\Combar}{C(\Ombar)}
\newcommand{\Ceombar}{C^1(\Ombar)}

\newcommand{\CombarNTmax}{C(\Ombar\times[0,\Tmax))}

\newcommand{\CzombarNTmax}{C^2(\Ombar\times[0,\Tmax)}

\newcommand{\Cninf}{C_0^\infty}

\newcommand{\Cinfty}{C^\infty}
\newcommand{\Cinf}{\Cinfty}

\newcommand{\ddt}{\frac{\rm d}{{\rm d}t}}
\newcommand{\ddr}{\frac{\rm d}{{\rm d}r}}

\newcommand{\norm}[2][]{\|#2\|_{#1}}

\newcommand{\weakto}{\rightharpoonup}
\newcommand{\weakstarto}{\overset{*}{\weakto}}

\newcommand{\wstarto}{\weakstarto}
\newcommand{\upto}{\nearrow}
\newcommand{\downto}{\searrow}
\newcommand{\embeddedinto}{\hookrightarrow}
\newcommand{\eingebettetin}{\embeddedinto}
\newcommand{\kompakteingebettetin}{\embeddedinto\embeddedinto}
\newcommand{\kpteingebettetin}{\kompakteingebettetin}

\newcommand{\limsuptTmax}{\limsup_{t\upto \Tmax}}

\newcommand{\heatgr}[1]{e^{#1 \Delta}}

\newcommand{\inv}{^{-1}}

\newcommand{\ohne}{\setminus}
\newcommand{\set}[1]{\{#1\}}
\newcommand{\setl}[1]{\left\{#1\right\}}

\newcommand{\weqsoln}{$W^{1,q}$-solution}

\newtheorem{theorem}{Theorem}
\newtheorem{lemma}[theorem]{Lemma}
\newtheorem{corollary}[theorem]{Corollary}
\newtheorem{proposition}[theorem]{Proposition}
\newtheorem{remark}[theorem]{Remark}

\newtheorem{thm}[theorem]{Theorem}
\newtheorem{cor}[theorem]{Corollary}

\newtheorem{defn}[theorem]{Definition}

\usepackage{fullpage}

\author{Johannes Lankeit\thanks{Institut f\"ur Mathematik, Universit\"at Paderborn, Warburger Str. 100, 33098 Paderborn, Germany; email: \mbox{johannes.lankeit@math.upb.de}}}
\title{Chemotaxis can prevent thresholds on population density}

\begin{document}
 \maketitle

\begin{abstract}
We define and (for $q>n$) prove uniqueness and an extensibility property of $W^{1,q}$-solutions to 
\begin{align*}
 u_t&=-\na(u\na v)+\kappa u-\my u^2\\
0&=\Lap v-v+u\\
\delny v\amrand &=\delny u\amrand=0 ,\qquad u(0,\cdot)=u_0,
\end{align*}
in balls in $\R^n$, which we then use to obtain a criterion guaranteeing some kind of structure formation in a corresponding chemotaxis system - thereby extending recent results of Winkler \cite{winkler_14_ctexceed} to the higher dimensional (radially symmetric) case.
\end{abstract}
%Keywords: chemotaxis, logistic source, blow-up, hyperbolic-elliptic system
%35K55 (primary) Nonlinear parabolic equations 
%35B44 (secondary) Blow-up 
%35A01 Existence problems: global existence, local existence, non-existence
%35A02 Uniqueness problems: global uniqueness, local uniqueness, non-uniqueness 
%35Q92  PDEs in connection with biology and other natural sciences 
%92C17  Cell movement (chemotaxis, etc.) 
%gibts nur bis 09: %35A07  Local existence and uniqueness theorems

\section{Introduction}
\label{sec:intro}
The Keller-Segel model
\begin{align*}
\label{eq:KS_classic}
u_t&=\na\cdot(D_2 u)-\na\cdot(D_1\na v)\\
v_t&=D \Lap v-k(v)v +u f(v) 
 \tag{KS}
\end{align*}
of chemotaxis has been introduced by Keller and Segel in \cite{keller_segel_70} to model the aggregation of bacteria (for instance, of the species {\it Dictyostelium discoideum}, with density denoted by $u$) in the presence of a signalling substance (cAMP, with density $v$) they emit in case of food scarceness. Their movement is governed by random diffusion and chemotactically directed motion towards higher concentrations of cAMP. 
The Keller-Segel model or variants thereof, as for example
\begin{align*}
 u_t&=\Lap u - \na\cdot(u\na v)\\
 \tau v_t&=\Lap v - v+u 
\end{align*}
where all functions appearing in \eqref{eq:KS_classic} have a simple %, prototypical 
form and 
diffusion of the signalling substance is assumed to occur fast (instantaneously if $\tau=0$), have been widely used and incorporated in more complicated models in the mathematical depiction of biological phenomena, ranging from pattern formation in {\it E. coli} colonies \cite{aotani_mimura_mollee_10} 
to angiogenesis in early stages of cancer \cite{szymanska_moralesrodrigo_lachowicz_chaplain_09} or HIV-infections \cite{stancevic_angstmann_murray_henry_13}. 
For a survey of the extensive mathematical literature on the subject see the survey articles \cite{hillen_painter_09} or \cite{horstmann_03,horstmann_04}.

Often the occurence of the desired structure formation is identified with the blow-up of solutions to the model in finite time, i.e. the existence of some finite time $T$ such that $\limsup_{t\upto T} \norm[L^\infty]{u(\cdot, t)}=\infty$ -- and the model -- both for $\tau=0$ and $\tau>0$ -- 
 is known to possess such solutions for every sufficiently large initial mass or in space dimensions larger than two, whereas in dimension $2$ for small initial mass all solutions exist globally in time and are bounded \cite{jaeger_luckhaus_92, herrero_velazquez_97,nagai_01,nagai_senba_yoshida_97}. 
Moreover, blow-up of solutions with large enough initial mass has been shown to be a generic phenomenon of the equation in some sense even for the parabolic-parabolic version of the system \cite{mizoguchi_winkler_13,winkler_11_blowuphigherdim}.

Another point of view is that blow-up is ``too much'' and biologically inadequate, at least in some situations. 
Then, for example, terms preventing blow-up are added, e.g. some logistic growth term (cf. for example the tumor models in \cite{andasari_gerisch_lolas_south_chaplain_11} or \cite{stinner_surulescu_winkler_13}), so that the model reads
\begin{align*}
 u_t&=\Lap u - \na\cdot(u\na v) +\kappa u-\my u^2\\
 \tau v_t&=\Lap v - v+u. % RB?.
\end{align*}
For this problem it is known \cite{tello_winkler_07,osaki_tsujikawa_yagi_mimura_02,winkler_10_boundedness} that classical solutions exist globally in time and are bounded if $n\leq 2$ or $\my$ is large (where for $\tau=0$, an explicit condition sufficient for this is $\my>\frac{n-2}n$).
For higher dimensions and small $\my$ the existence or non-existence of exploding solutions is unknown. As \cite{winkler_11_blowupdespiteloggrowth} seems to indicate, superlinear absorption does not necessarily imply global existence.

The important question is: To what extent does the logistic term render the chemotaxis-term innocuous? Does there still emerge some structure?
Recently this question has been answered affirmatively by Winkler \cite{winkler_14_ctexceed} in the one-dimensional case:
If the death rate $\my$ is small enough $(0<\my<1)$, then there is some criterion on (the $L^p$-norm with $p>\frac{1}{1-\my}$ of) the initial data that ensures the existence of some time up to which any threshold of the population density will be surpassed - as long as the bacteria do not diffuse to fast.
% bacteria->amobae?
Of course, the biologically relevant situation is not that of only one space-dimension. 
With the present paper we give an answer to the question whether this phenomenon is 
restricted to this case or if it also occurs in higher dimensions.

We shall confine ourselves to the prototypical radially symmetric setting and in the end obtain 
\begin{thm}
\label{thm:main}
Let $\Omega\subset\Rn$ be a ball.  Let $\kappa\geq 0, \my\in(0,1)$. Then for all $p>\frac1{1-\my}$ there exists $C(p)>0$ satisfying the following:
Whenever $q>n$ and $u_0\in \Weqom$ is nonnegative, radially symmetric and compatible and such that 
\[
 \norm[\Lpom]{u_0}>C(p)\max\setl{\frac1{|\Om|}\intom u_0,\frac\kappa\my},
\]
there is $T>0$ such that to each $M>0$ there corresponds some $\eps_0(M)>0$ with the property that for any $\eps\in(0,\eps_0(M))$ one can find $t_\eps\in(0,T)$ and $x_\eps\in\Om$ such that the solution $(u,v)$ of 
\begin{align}
\label{eq:epsprob}
 u_t&=\eps\Lap u-\na(u\na v)+\kappa u-\my u^2\\
0&=\Lap v-v+u\notag\\
\delny v\amrand&=\delny u\amrand=0,\qquad u(0,\cdot)=u_0,\notag
\end{align}
in $\Om\times(0,\Tmax)$, where $\Tmax\in(0,\infty]$ is its maximal time of existence, satisfies $u(x_\eps,t_\eps)>M$.
\end{thm}

For this purpose we set out to find estimates finally leading to the crucial extensibility criterion \eqref{eq:ext} for solutions of the  ``$\eps=0$-limit'' model 
\begin{align}
\label{eq:limprob}
 u_t&=-\na(u\na v)+\kappa u-\my u^2\\
0&=\Lap v-v+u\notag\\
\delny v\amrand&=\delny u\amrand=0,\qquad u(0,\cdot)=u_0\notag
\end{align}
of \eqref{eq:epsprob} in $\Om\times(0,T)$ for some $T>0$.
The extensibility criterion is analogous to (1.6) of \cite{winkler_14_ctexceed}, that in turn is built upon estimates, some of which heavily rely on one-dimensionality of the problem. 
Cornerstone of our analysis %and main point of deviation from \cite{winkler_14_ctexceed} 
therefore will be Section \ref{sec:estimate}, where we craft 
the inequality which also allows for higher-dimensional and therefore more realistic scenarios.

We will introduce our concept of solutions of \eqref{eq:limprob} in Definition \ref{def:soln} and show their uniqueness -- if $u_0\in\Weqom$ for some $q>n$ -- in Theorem \ref{thm:unique} and the existence of radially symmetric solutions that can be approximated by solutions of \eqref{eq:epsprob} in Theorem \ref{thm:existence}. 

In contrast to the one-dimensional case, we are confronted with the challenge that we cannot, in general, rely on the existence of global classical bounded solutions to the approximate problem.
Hence we prepare these results by finding a common existence time of such solutions -- regardless of the value of $\eps$ (Theorem \ref{thm:commonexistencetime}). 

After collecting some additional boundedness property in Lemma \ref{thm:timederivative},
 we can by a limiting procedure (Lemma \ref{thm:kgztosolnifbd}) turn to solutions to \eqref{eq:limprob}.

If then $\my$ is large enough, a global, in some cases even bounded solution is guaranteed to exist [Prop. \ref{thm:limprobglobal}]. 
However, if this is not the case, any radial solution to \eqref{eq:limprob} with somehow ($L^p$-)large enough initial mass blows up in finite time (Theorem \ref{thm:blowup}).
In combination with the fact that solutions to \eqref{eq:limprob} can be obtained as limits of solutions to \eqref{eq:epsprob}, this yields the announced theorem about nonexistence of thresholds to population density:
If $\my<1$ and $\norm[L^p]{u_0}$ (for $p>\frac1{1-\my}$) is large enough, before some time $T$ any threshold on the population density will be exceeded at least at one point by any population that diffuses slowly enough.
%Note, however, that in spite of the title of \cite{winkler_14_ctexceed}, it is {\bf not} a carrying capacity being exceeded here, as this would mean an $\norm[L^1]{\cdot}$-norm becoming large. \rot{Da bin ich andere Meinung, denn die
%carrying capacity ist f\"ur mich eine kritische {\em Dichte}.}

After the following short section which recalls a few basic properties of solutions to the second equation in \eqref{eq:epsprob} and equation \eqref{eq:epsprob} with $\eps>0$, in Section \ref{sec:parabolic_elliptic} we focus our attention on existence of solutions to \eqref{eq:epsprob} and estimates yielding a common existence time (Theorem \ref{thm:commonexistencetime}) as well as preparing for compactness arguments (by the estimates of Lemma \ref{thm:timederivative} and Corollary \ref{thm:theestimate}).

Section \ref{sec:limprob} will be devoted to definition, uniqueness, existence, estimates and a blow-up result for solutions to \eqref{eq:limprob}, followed in Section \ref{sec:nothreshold} by, finally, the proof of the ``no threshhold'' theorem \ref{thm:main}.

Throughout the article, we assume that $\Om\subset\R^n$ is a bounded domain with smooth boundary. Often we will speak of radially symmetric (or, for short, radial) functions. In this case $\Om=B_R(0)$ is to be understood to be a ball centered in the origin and we will interchange $u(x,t), x\in\Om$, and $u(r,t), r=|x|\in [0,R)$. 

Occasionally, we will abbreviate $\Om\times(0,T)=:Q_T$ for $T>0$.
 We will often identify $u(t)=u(\cdot,t)$ and for the sake of brevity write $u$ instead of $u(x,t)$ or $u(x,s)$. 
%Unless stated otherwise, we assume $kappa\geq 0, \my>0, q>n$.
\section{Preliminaries: The elliptic equation}
\label{sec:secondeq}
In the proofs we will mainly be concerned with $u$, therefore it would be desirable to estimate various terms involving the solution $v$ of
\begin{equation}
 \label{eq:v}
 -\Lap v+v=u \quad \text{in }\Om,\qquad\qquad \delny v\amrand=0 
\end{equation}
or its derivatives in terms of $u$. The following lemmata will be the tools to make this possible:

\begin{lemma}
\label{thm:vlpleu}
 Let $v$ solve \eqref{eq:v}. Then for all $p\in[1,\infty]$, 
\[
 \norm[\Lpom]{v}\leq\norm[\Lpom]u
\]
\end{lemma}
\begin{proof}
 As in \cite[Lemma 2.1]{winkler_14_ctexceed}, for $p\in(1,\infty)$, testing the equation by $v(v^2+\eta)^{\frac p2-1}$ as $0<\eta\to 0$ yields this estimate, which can then be extended to $p\in\set{1,\infty}$ by limiting procedures.
\end{proof}

If we rather prefer estimates involving the $L^1$-norm of $u$, that is possible as well:
\begin{lemma}
 \label{thm:vlppeleule}
 For all $p>0$, $\eta>0$ there exists $C(\eta,p)>0$ such that whenever $u\in\Combar$ is nonnegative, the solution $v$ of \eqref{eq:v}
satisfies 
\[
 \intom v^{p+1}\leq \eta \intom u^{p+1}+C(\eta,p)\left(\intom u\right)^{p+1}.
\]
\end{lemma}
\begin{proof}(as in \cite[Lemma 2.2]{winkler_14_ctexceed})
 Multiply \eqref{eq:v} by $v^p$, integrate over $\Om$ and use Young's inequality:
\[
 p\intom v^{p-1}|\na v|^2+\intom v^{p+1}=\intom uv^p\leq \frac1{p+1}\intom u^{p+1}+\frac p{p+1}\intom v^{p+1},
\]
therefore 
\[
 p\intom v^{p-1}|\na v|^2\leq \frac1{p+1}\intom u^{p+1}.
\]
Because $\na(v^\frac{p+1}2)=\frac{p+1}2 v^\frac{p-1}2\na v$ and $v^{p-1}|\na v|^2=\frac 4{(p+1)^2}(\na (v^\frac{p+1}2))^2$ we obtain 
\[
 \frac{4p}{p+1}\intom|\na (v^\frac{p+1}2)|^2\leq \intom u^{p+1}.
\]

As now 
\[
 \Wezom\kpteingebettetin\Lzom\eingebettetin L^{\frac2{p+1}}(\Om),
\]
by Ehrling's lemma there is $\ctilde_1=\ctilde_1(\eta,p)>0$ (and hence $c_1>0$) such that for all $\psi\in\Wezom$
\begin{align*}
 \norm[L^2]{\psi}^2&\leq \frac {4p}{p+1}\eta\norm[\Wezom]{\psi}^2+\ctilde_1\norm[L^{\frac2{p+1}}(\Om)]{\psi}^2
\leq\frac{4p}{p+1} \eta\norm[\Lzom]{\na \psi}^2+ c_1\norm[L^\frac{2}{p+1}(\Om)]{\psi}^2.
\end{align*}
Applying these two inequalities to $v^\frac{p+1}2$ and using Lemma \ref{thm:vlpleu} for $p=1$, we arrive at
\begin{align*}
 \intom v^{p+1}&\leq \frac{4p}{p+1}\eta\intom |\na(v^\frac{p+1}2)|^2+c_1\left(\intom v\right)^{p+1}
\leq \eta\intom u^{p+1}+c_1\left(\intom u\right)^{p+1}. \qedhere
\end{align*}
\end{proof}

We also recall useful facts on maximal regularity for elliptic PDEs:
\begin{lemma}
\label{thm:elliptregularity}
 For $q\geq 1, \alpha>0$, there is a constant $C>0$ such that any (classical) solution $v$ of 
\[
 \Lap v-v=f\quad \text{in }\Om,\qquad\qquad \delny v\amrand =0 
\]
satisfies $\norm[\Wzqom]{v}\leq C \norm[\Lqom]{f}$ %Friedman, Thm 19.1, bzw. Problem (3), p. 76
and $\norm[C^{2+\alpha}(\Om)]{v}\leq C \norm[C^\alpha(\Om)]{f}$.
\end{lemma}
\begin{proof}
 \cite[ch. 19]{friedman_08}
 \cite[Thm. 6.30]{gilbarg_trudinger_01}
\end{proof}

\begin{lemma}
\label{thm:nonneg}
 Let $v$ solve \eqref{eq:v} for a nonnegative right-hand side $u$. Then also $v$ is nonnegative.
\end{lemma}
\begin{proof}
 This is a consequence of the elliptic maximum principle.
\end{proof}

\section{Parabolic-elliptic case}
\label{sec:parabolic_elliptic}
\subsection{Existence}
We prepare the following two lemmata with this estimate from \cite[Lemma 1.3 iv)]{winkler_10_aggregationvs} about the (Neumann) heat semigroup:
\begin{lemma}
\label{thm:heatestimate}
 Let $1<q\leq p\leq \infty$. Then there exists $C>0$ such that for all $t>0$ and for all $w\in(\Lqom)^n$
\[
 \norm[\Lpom]{\heatgr t\na\cdot w}\leq C(1+t^{-\frac12-\frac n2(\frac1q-\frac1p)})\norm[\Lqom]{w}.
\]
\end{lemma}
\begin{proof}
 \cite[Lemma 1.3 iv)]{winkler_10_aggregationvs}. Although the lemma in that article is stated only for $q<\infty$, the proof actually already covers the case $q=\infty$, because $\Cinf(\Om)$ is dense in $L^1(\Om)$.
\end{proof}

\label{sec:epsexistence}
\label{sec:epslocalexistence}
One of the first steps in dealing with solutions of \eqref{eq:epsprob} is to show that they exist, at least locally. Let us briefly give the corresponding fixed point arguments.

\begin{lemma}
\label{thm:epslocalexistence}
 Let $u_0\in\Combar,\eps>0,\kappa\geq 0,\my>0$. Then there is $\Tmax\in(0,\infty]$ such that \eqref{eq:epsprob} has a unique classical solution on 
 $\Ombar\times(0,\Tmax)$ and 
\[
 \limsuptTmax \norm[\Liom]{u(\cdot,t)}=\infty \text{ or } \Tmax=\infty.
\]
\end{lemma}
\begin{proof}
 For $u\in\Combar$ denote by $v_u$ the solution of 
 \[
  0=\Lap v+u-v\quad \text{in }\Om,\qquad\qquad \delny v\amrand=0.
 \]
 Let $R>2\norm[\Weqom]{u_0}$ be given. Fix constants $C_1$ as in Lemma \ref{thm:heatestimate}, $C_2$ and the function $C\colon [0,1]\to \R$ such that 
\[
 \norm[\Lqom]{\na v}\leq C_2\norm[\Liom]{u}, \qquad C(t)=C_1\intnt 1+(\eps(t-s))^{-\frac12-\frac n{2q}}ds, t>0, 
\]
and note that $C$ is monotone and continuous with $C(0)=0$. Choose $\Ttilde\in(0,1)$ such that 
\[
 (\kappa+2\my R)\Ttilde +2R C_2 C(\Ttilde)<\frac12
\]
and let $T\in(0,\Ttilde)$.
For $t\in[0,T]$ define 
\[
 \Phi(u)(\cdot,t):=\heatgr{\eps t} u_0-\intnt \heatgr{\eps(t-s)} \na\cdot(u(s)\na v_u(s))ds+\kappa\intnt\heatgr{\eps(t-s)}u(s)ds-\my\intnt \heatgr{\eps(t-s)}u^2(s) ds.
\]
Then $\Phi\col C(\Qbar_T)\to C(\Qbar_T)$ 
 is well-defined and, in fact, even $\Phi(u)\in \Cinf(Q_T)$. 
In addition, $\Phi$ is a contraction in $M:=\setl{f\in C(\Qbar_T); \norm[\LiQT]f\leq R}$, as can be seen as follows:
\begin{align*}
  \norm[\LiQT]{\Phi(u)-\Phi(\utilde)}&\leq 
\sup_{0<t<T}\intnt\norm[\Liom]{\heatgr{\eps(t-s)} \na(u(s)\cdot\na v_u(s)-\utilde(s)\na v_{\utilde}(s))}ds\\
&\qquad+\sup_{0<t<T} \intnt\norm[\Liom]{\heatgr{\eps(t-s)}(\kappa(u(s)-\utilde(s))+\my(u^2(s)-\utilde^2(s)))}ds\\%+\my\intnt\norm[\Liom] &\leq \sup_{0<t<T} C_1\intnt(1+(\eps(t-s))^{-\frac12-\frac n{2q}})\norm[\Lqom]{(u(s)-\utilde(s))\na v_u(s)+\utilde(s)\na v_{u-\utilde}(s)}ds+\\
&\qquad\qquad\qquad\qquad\qquad\qquad\qquad\qquad\qquad\qquad\qquad+T(\kappa+2R\my)\norm[\LiQT]{u-\utilde}\\
& \leq C(T)(\norm[\LiQT]{u-\utilde}C_2R+RC_2\norm[\LiQT]{u-\utilde})+T(\kappa+2 \my R)\norm[\LiQT]{u-\utilde}\\
&\leq \frac12\norm[\LiQT]{u-\utilde}.
\end{align*}
Furthermore $\Phi$ maps $M$ to $M$ as well:
\begin{align*}
 \norm[\LiQT]{\Phi(u)}&\leq\norm[\LiQT]{\Phi(u)-\Phi(0)}+\norm[\LiQT]{\Phi(0)}\leq \frac12\norm[\LiQT]{u}+\norm[\Liom]{u_0}\leq R.
\end{align*}
With the aid of Banach's fixed point theorem, this procedure yields a solution on $(0,T)$. 
Successively employing the same reasoning on later time intervals (then with different $u_0$ and possibly larger $R$) the existence of a solution on a maximal time interval $(0,\Tmax)$ is obtained where either $\Tmax=\infty$ or $\limsuptTmax \norm[\Liom]{u(\cdot,t)}=\infty$.
\end{proof}

\subsection{$L^p$-bounds and global existence}

Bounds on $L^p$-norms are of great utility, not only for the deduction of global existence. 
A standard testing procedure (see also \cite{winkler_14_ctexceed}) yields 
\begin{lemma}
 \label{thm:ddtulp}
 Let $\kappa\geq 0$, $\my>0$, $u_0\in\Combar$ nonnegative. Let $(u,v)$ solve $\eqref{eq:epsprob}$ classically in $\Om\times(0,T)$ for some $T>0, \eps>0$. 
 Then for $p\geq 1$ and on the whole time interval $(0,T)$, we have
 \begin{equation}
  \label{eq:dreieins}
  \ddt \intom u^p+p(p-1)\eps\intom u^{p-2}|\na u|^2\leq p\kappa\intom u^p -(1-p+\my p)\intom u^{p+1}.
 \end{equation}
\end{lemma}
\begin{proof}
 Multiplication of the first equation of \eqref{eq:epsprob} by $u^{p-1}$ and integration by parts yield
 \[
  \frac1p\ddt\intom u^p+(p-1)\eps\intom u^{p-2}|\na u|^2=(p-1)\intom u^{p-1}\na u\na v+\kappa \intom u^p-\my\intom u^{p+1}.
 \]
 Another integration by parts in combination with the second equation of \eqref{eq:epsprob} and Lemma \ref{thm:nonneg} show
 \[
  (p-1)\intom u^{p-1}\na u \na v =-\frac{p-1}p\intom u^p \Lap v=\frac{p-1}p\intom u^p(u-v)\leq \frac{p-1}p\intom u^{p+1},
 \]
 which gives formula \eqref{eq:dreieins}.
\end{proof}
This estimate directly leads to the following bound on $L^p$-norms of $u$.
\begin{corollary}
\label{thm:ulpbounded}
  Let $\kappa\geq 0, \my>0, u_0\in \Combar$ nonnegative, suppose $(u,v)$ is classical solution of \eqref{eq:epsprob} in $Q_T$ for some $T>0,\eps>0$. Let $p\in[1,\frac{1}{(1-\my)_+})$.\\
 Then for all $t\in[0,T)$,  \[\intom u^p(\cdot,t)\leq \max\setl{\intom u_0^p, \left(\frac{p\kappa}{1-(1-\my)p}\right)^p|\Om|}.\]
\end{corollary}
\begin{proof}
 An application of H\"older's inequality gives $\intom u^{p+1}\geq |\Om|^{-\frac1p}(\intom u^p)^{\frac{p+1}p}$ and transforms \eqref{eq:dreieins} into the differential inequality
 \[
  y'(t)\leq p\kappa y(t)-(1-p+\my p)|\Om|^{-\frac1p}(y(t))^{1+\frac1p}), \quad t\in(0,T),
 \]
 for $y=\intom u^p$. An ODE-comparison then yields the result.
\end{proof}

\begin{corollary}
 \label{thm:n2bounded}
 Let $\kappa\geq 0$, $q>n$, $u_0\in\Combar$. If $\my>\frac{n-2}{n}$, the solutions of \eqref{eq:epsprob} are global.
\end{corollary}
\begin{proof}
This arises from the bounds in Corollary \ref{thm:ulpbounded} by arguments that can be found in Lemmata 2.3 and 2.4 of \cite{tello_winkler_07}.
% If $\my>\frac{n-2}{n}$, a choice $p>\frac n2$ is possible in Corollary \ref{thm:ulpbounded}. Lemma \ref{thm:regularityjump} then rules out the possibility of 
%$\limsup_{t\upto T} \norm[\Liom]{u(\cdot,t)}=\infty$ for any finite $T$.
\end{proof}

If even $\my\geq 1$, bounds can be given in a more explicit form and independent of $\eps$.

\begin{lemma}
 \label{thm:epsprobbounds}
 Let $\kappa\geq 0, \my \geq 1, u_0\in\Combar, u_0\geq 0, u_0\neq 0$, $(u_\eps,v_\eps)$ classical solution of \eqref{eq:epsprob} in $\Om\times(0,\infty)$ for $\eps>0$. 
 Then, for all $t>0$,
\[
 \norm[\Liom]{u_\eps(\cdot,t)}\leq\begin{cases}
  \frac{\kappa}{\my-1}(1+(\frac{\kappa}{(\my-1)\norm[\Liom]{u_0}}-1)e^{-\kappa t})\inv&\kappa>0,\my>1,\\
  \frac{\norm[\Liom]{u_0}}{1+(\my-1)\norm[\Liom]{u_0}t}&\kappa=0,\my>1,\\
  \norm[\Liom]{u_0}e^{\kappa t}&\kappa>0, \my=1,\\
  \norm[\Liom]{u_0}&\kappa=0, \my=1.
  \end{cases}
\]
\end{lemma}
\begin{proof}
(Cf. \cite[Lemma 4.6]{winkler_14_ctexceed}.) This can be obtained by comparison with the solution $y$ of 
\[
 y'(t)=\kappa y(t)-(\my-1)y^2(t),\;\;t>0,\qquad y(0)=\norm[\Liom]{u_0}.\qedhere
\] 
\end{proof}

\subsection{Radial solutions}
In the following sections we will restrict ourselves to the prototypical radially symmetric situation. 
In this case, equations \eqref{eq:epsprob} can be rewritten in the form 
\begin{align}
\label{eq:radial1}
 u_t&=\eps u_{rr}+\eps\frac{N-1}{r} u_r -u_rv_r-uv_{rr}-\frac{N-1}{r}uv_r +\kappa u -\my u^2\\
\label{eq:radial2}
 0&=v_{rr}+\frac{N-1}{r}v_r-v+u.
\end{align}
 We begin by preparing an inequality for the derivative of $v$. Gained by the radial symmetry, it will be one of the most important tools for the calculations preparing the estimation of $\norm[\Lqom]{\na u}$ in terms of $\norm[\Liom]{u}$.

\begin{lemma}
\label{thm:radialestimate}
 Let $(u,v)$ 
 be a radially symmetric nonnegative classical solution of \eqref{eq:epsprob}. Then for $r\in[0,R], t>0$,
\begin{equation}
 \label{eq:radialtrick}
  v_r(r,t)\leq \frac1Nr\norm[\Liom]{u(\cdot,t)}.
\end{equation}
\end{lemma}
\begin{proof}
 Fix $t>0$. Equation \eqref{eq:radial2} can also be written in the form $ \frac1{r^{N-1}}(r^{N-1}v_r)_r = v-u$ and implies
 \[
  (r^{N-1}v_r)_r=r^{N-1}(v-u),
 \]
 hence 
 \[
  r^{N-1} v_r(r,t)= 0+\intnr \rho^{N-1} (v(\rho,t)-u(\rho,t)) d\rho \leq \norm[\Liom]{u(t)} \intnr\rho^{N-1}d\rho = \frac1N r^N \norm[\Liom] {u(t)}, 
 \]
 which leads to \eqref{eq:radialtrick}. 
\end{proof}

\subsection{Compatibility}
\label{sec:compatibility}
We say that a function $u_0$ satisfies the %first?
compatibility criterion (or, for short, that $u_0$ is compatible) if $u_0\in C^1(\Ombar)$ and $\delny u_0\amrand=0$.
If functions with this property are used as initial condition in parabolic problems, the solutions they yield have bounded first [spatial] derivatives on a time interval containing $0$ (\cite{ladyzhenskaya_solonnikov_uraltseva_91}).
This will be important in the derivation of the crucial estimate of $\intom |\na u|^q$ for solutions $u$ of \eqref{eq:epsprob} (Lemma \ref{thm:theestimate}) in terms of $\norm[L^q]{\na u_0}$ instead of only $\norm[L^q]{\na u(\cdot,\tau)}$ for arbitrary small $\tau>0$.

%In the one-dimensional case this problem does not occur because also derivatives of solutions to the HeatEquation (with Neumann BC) satisfy the heat equation (with Dirichlet BC) and therefore aro known to be continous with values in $L^q$ -- that implies $t\mapsto\norm[q]{\na u(t)}$ is continuous also in $0$ without the assupmtion of especially nice initial data.
At first we show that any function $u_0\in\Weqom$ can be approximated by compatible functions preserving all kind of 'nice properties':

\begin{lemma}
 \label{thm:compatibleapproximation}
 Let $q>n$, $u_0\in\Weqom$ be radially symmetric and nonnegative, let $\eps>0$. There is $\utilde_0\in\Ceombar$ with $\delny\utilde_0\amrand=0$ such that $\norm[\Weqom]{u_0-\utilde_0}<\eps$ and also $\utilde_0$ is radial and nonnegative.
\end{lemma}
\begin{proof}
 Given $\eps>0$ consider the standard mollifications $\eta^\epstilde\star \uhat_0 $ (cf. \cite[C.4]{evans_98}) of 
 \[
  \uhat_0:=u_0\one_{B_{R-\frac\delta2}}+u_0\left(R-\frac\delta2\right)\,\one_{\R^n\ohne B_{R-\frac\delta2}}
 \]
 in $\Om$, where $\one$ denotes the characteristic function of a set, for an appropriate, small choice of $\delta$ and $\epstilde$.
\end{proof}

% \begin{proof}
% Due to $q>n$, $u$ is continuous. Let $M:=\sup_\Om |u|$. By the monotone convergence theorem of Beppo-Levi, 
% \[
%  \norm[\Weq(B_{R-\delta})]{u}\to \norm[\Weqom] \quad \text{as }\delta\to 0. 
% \]
% Now choose $\delta>0$ so small that
% \[
%  \max\setl{\norm[\Weq(B_R)]{u}-\norm[\Weq(B_{R-2\delta})]{u},\int_{B_R\ohne B_{R-\delta}} M}\leq \frac\eps4.
% \]
% Consider the mollification $\eta^\epstilde\star\uhat=\utilde$ with 
% \begin{equation}
%  \label{eq:defuhat}
%  \uhat:=u\one_{B_{R-\frac\delta2}}+u(R-\frac\delta2)\one_{\R^n\ohne B_{R-\frac\delta2}}
% \end{equation}
% and $\eta^\epstilde$ denoting a standard mollifier as in \cite[C.4]{evans_98}. 
% Then 
% \[
%  \na (\eta^\epstilde\star\uhat)=\eta^\epstilde\star[\na u\cdot\one_{B_{R-\frac\delta2}}]. 
% \]
% And if in \eqref{eq:defuhat}, $\epstilde>0$ is chosen so small that 
% \[
%  \epstilde\leq\frac\delta2\;\text { as well as } \norm[\Weq(B_{R-\delta})]{u-\eta^\epstilde\star\uhat}\leq\frac\eps4
% \]
% then 
% \[
%  \delny \utilde\amrand=0
% \]
% and, furthermore, 
% \begin{align*}
%  \norm[\Weq(B_R)]{\utilde-u}&\leq \norm[\Weq(B_{R-\delta})]{u-\utilde}+\norm[\Weq(B_R\ohne B_{R-\delta})]{u} +\norm[\Weq(B_R\ohne B_{R-\delta})]{\utilde}\\
% &\leq \frac\eps4+\frac\eps4+\norm[L^q(B_R\ohne B_{R-\delta})]M+\norm[L^q(B_R\ohne B_{R-\delta})]{\utilde}\\
% &\leq \frac{3\eps}4+\underbrace{\norm[L^q(B_R\ohne B_{R-2\delta})]{\na \uhat}}_{=\norm[L^q(B_R\ohne B_{R-2\delta})]{\na u}}\leq \eps.
% \end{align*}
% And of course, $\utilde\geq 0$, because $u\geq 0$. Also symmetry is preserved.
% \end{proof}

\subsection{The most important estimate}
\label{sec:estimate}
In this section we are going to derive an inequality 
which shows that we can control the $\norm[\Weqom]{\cdot}$-norm of solutions to \eqref{eq:epsprob} by their $\Liom$-norm.
For the following computation we define, for $\eta>0$, 
\[
  \Phi_\eta(s):=(s^2+\eta)^{\frac q2}, \quad s\in\R,
\] 
and compute
 \[
  \Phi_\eta'(s)= q s(s^2+\eta)^{\frac q2-1},
 \]
 which implies 
 \begin{equation}
  \label{eq:sphietastrich}
  s\Phi_\eta'(s)\leq q\Phi(s) \qquad \textrm{as well as}\qquad s\Phi_\eta'(s)\geq 0 
 \end{equation}
 for $s\in\R$ and 
 \begin{align}
\label{eq:phietazweistrich}
  \Phi_\eta''(s)
 &= q ((q-1) s^2+\eta )(s^2+\eta)^{\frac q2-2}\geq 0,\qquad s\in\R.
 \end{align}
 In preparation for later calculations we also note that for $a,s\in\R$
\begin{equation}
 \label{eq:phietaminussphietastrich}
 \Phi_\eta(s)-as\Phi_\eta'(s)=(1-aq)s^2(s^2+\eta)^{\frac q2-1}+\eta(s^2+\eta)^{\frac q2-1}.
\end{equation}

%replacing wk, Lem. 3.5 and Corollary 3.6
\begin{lemma}
\label{thm:estimate}
  Let $\kappa\geq 0, \my>0, q>n, T>0$.\\
 For any radial classical solution $u$ of \eqref{eq:epsprob} in $\Om\times(0,T)$ with radial initial data $u_0\in\Weqom$, and arbitrary $\tau\in(0,T)$, $t\in(\tau, T)$, we have (with $K$ as $C$ from Lemma \ref{thm:elliptregularity})
\begin{align*}
 \int_{\Om}& \Phi_\eta\left(|\na u(\cdot,t)|\right)\leq \int_{\Om} \Phi_\eta(|\na u(\cdot, \tau)|)\\
&+ \inttaut\left((5q+\frac{q-2}q \eta)\norm[\Liom]{u}+\kappa q\right)\int_{\Om} \Phi_\eta\left(|\na u|\right)+|\Om|\inttaut\left(K\norm[\Liom]{u}^{1+q}+\frac {2\eta}q \norm[\Liomnt]{u}\right)
\end{align*}
\end{lemma}
\begin{proof}
 Denote $\Om_\delta=\Om\ohne B_\delta(0)$ and let $0<\tau<t<T$. \\
 Note that on $\Om_\delta\times (\tau,t)$ all derivatives of $u$ appearing in the following calculation are smooth and bounded, and we can change the order of integration and differentiation to start with
 \begin{align*}
 \int_{\Om_\delta} \Phi_\eta\left(|\na u(\cdot,t)|\right) - \int_{\Om_\delta} \Phi_\eta(|\na u(\cdot, \tau)|) 
%by the usual abuse of notation, $\Phi_\eta$ denotes functions on $\Rn$ as well as (radial setting) on $\R$.
=\inttaut\intdlR r^{N-1}\Phi_\eta'(u_r) u_{rt}.
\end{align*}
Here we use equation \eqref{eq:radial1} for $u_t$:
\begin{align*}
\int_{\Om_\delta}& \Phi_\eta\left(|\na u(\cdot,t)|\right) - \int_{\Om_\delta} \Phi_\eta(|\na u(\cdot, \tau)|) \\
&= \inttaut\intdlR r^{N-1}\Phi_\eta'(u_r) \left(\eps u_{rr}+\eps\frac{N-1}{r} u_r -u_rv_r-uv_{rr}-\frac{N-1}{r}uv_r +\kappa u -\my u^2\right)_r\\
&= \eps \inttaut\intdlR r^{N-1}u_{rrr}\Phi_\eta'(u_r)+\eps\inttaut\intdlR r^{N-2}(N-1)\ddr(\Phi_\eta(u_r))-\eps \inttaut\intdlR r^{N-3}(N-1)u_r \Phi_\eta'(u_r)\\
&\;\;-\inttaut\intdlR r^{N-1}v_ru_{rr}\Phi_\eta'(u_r)-2\inttaut\intdlR r^{N-1} v_{rr}u_r \Phi_\eta'(u_r)-\inttaut\intdlR  r^{N-1}v_{rrr}u\Phi_\eta'(u_r)\\
&\;\;+\intdlR r^{N-3}(N-1)v_ru\Phi_\eta'(u_r)-\inttaut\intdlR r^{N-2}(N-1)v_ru_r \Phi_\eta'(u_r)-\inttaut\intdlR r^{N-2}(N-1)v_{rr}u\Phi_\eta'(u_r)\\
&\;\;+\kappa \inttaut\intdlR r^{N-1}u_r\Phi_\eta'(u_r)-2\my \inttaut\intdlR r^{N-1} uu_r \Phi_\eta'(u_r)=:I_1+I_2+\ldots+I_{11}.
\end{align*}

Now we integrate by parts twice in the first term 
\begin{align*}
 I_1&=-\eps \inttaut\intdlR r^{N-1} u_{rr}\Phi_\eta''(u_r)u_{rr}-\eps  \inttaut\intdlR r^{N-2} u_{rr}\Phi_\eta'(u_r)(N-1)+\eps \inttaut r^{N-1} u_{rr}\Phi_\eta'(u_r)\vonbis \delta R\\
&\leq 0+\eps(N-1)(N-2)\inttaut\intdlR r^{N-3}\Phi_\eta(u_r)-\eps\inttaut r^{N-2}\Phi_\eta(u_r)(N-1)\vonbis \delta R+\eps \inttaut r^{N-1} u_{rr}\Phi_\eta'(u_r)\vonbis \delta R,
\end{align*}
where we also used \eqref{eq:phietazweistrich}, and once in the second integral
\[
 I_2=+ \eps(N-1)\inttaut r^{N-2}\Phi_\eta(u_r)\vonbis \delta R -\eps(N-1)(N-2)\inttaut\intdlR r^{N-3}\Phi_\eta(u_r). 
\]
Upon addition, some of these summands vanish and estimating $I_3\leq 0$ by \eqref{eq:sphietastrich}, we obtain
\[
 I_1+I_2+I_3\leq\eps\inttaut r^{N-1}u_{rr}\Phi_\eta'(u_r)\vonbis \delta R=-\eps\inttaut r^{N-1}u_{rr}\Phi_\eta'(u_r)\big|_{r=\delta},
\]
because $u_r(R,\ttilde)=0$ for all $\ttilde\in(0,T)$.\\
Also the next term can be rewritten by integration by parts 
and using $v_r(R,\ttilde)=0$ for $\ttilde\in (0,T)$.
\begin{align*}
 I_4
 &=\inttaut r^{N-1}v_r\Phi_\eta(u_r)\big|_{r=\delta} +(N-1)\inttaut\intdlR r^{N-2} v_r \Phi_\eta(u_r)+\inttaut\intdlR r^{N-1}v_{rr} \Phi_\eta(u_r).
\end{align*}
Inserting \eqref{eq:radial2} to express $v_{rrr}$ in $I_6$ differently, we obtain (among others) terms to cancel out $I_7$ and $I_9$:
\begin{align*}
 I_6
&=-\inttaut\intdlR r^{N-1}v_ru\Phi_\eta'(u_r)
+\inttaut\intdlR r^{N-1}uu_r\Phi_\eta'(u_r)
-I_9-I_7.
\end{align*}

Together with the trivial observation that $I_{11}\leq 0$ by \eqref{eq:sphietastrich}, these estimates and reformulations give 
\begin{align*}
 \int_{\Om_\delta}& \Phi_\eta\left(|\na u(\cdot,t)|\right) - \int_{\Om_\delta} \Phi_\eta(|\na u(\cdot, \tau)|)
\leq -\eps\inttaut r^{N-1}u_{rr}\Phi_\eta'(u_r)\big|_{r=\delta} +\inttaut r^{N-1}v_r\Phi_\eta(u_r)\big|_{r=\delta}\\
&\;\;+(N-1)\inttaut\intdlR r^{N-2} v_r \Phi_\eta(u_r)+\inttaut\intdlR r^{N-1}v_{rr} \Phi_\eta(u_r)\\
&\;\;-2\inttaut\intdlR r^{N-1} v_{rr}u_r\Phi_\eta'(u_r)
-\inttaut\intdlR r^{N-1}v_ru\Phi_\eta'(u_r)\\
&\;\;+\inttaut\intdlR r^{N-1} uu_r\Phi_\eta'(u_r)
-\inttaut\intdlR r^{N-2}(N-1) v_ru_r \Phi_\eta'(u_r)+\kappa \inttaut\int_\delta^R r^{N-1} u_r \Phi_\eta'(u_r) .
\end{align*}

Passing to the limit $\delta\downto 0$ by boundedness of $u_r, u_{rr}, v_r$ on $(\tau,t)$ and the dominated convergence theorem
%beachte auch: Sogar v_rr auf Nullumgebung (wo anders brauchen wirs für (DOM) gar nicht) beschränkt, da cts.
%which is possible, because all functions are bounded
we arrive at %\Phi_\eta'(u_r(R))=0, v_r(R)=0
\begin{align*}
 \int_{\Om}& \Phi_\eta\left(|\na u(\cdot,t)|\right) - \int_{\Om} \Phi_\eta(|\na u(\cdot, \tau)|)
=(N-1)\inttaut\intnR r^{N-2}v_r\left[\Phi_\eta(u_r)- u_r\Phi_\eta'(u_r)\right] \\
&+ \inttaut\intnR r^{N-1}v_{rr} \left[\Phi_\eta(u_r)- 2 u_r\Phi_\eta'(u_r)\right]
-\inttaut\intnR r^{N-1}v_ru\Phi_\eta'(u_r)\\
&\;\;+\inttaut\intnR r^{N-1}uu_r\Phi_\eta'(u_r)+\kappa q\inttaut\int_{\Om} \Phi_\eta\left(|\na u|\right)\;=\;I_A+I_B+I_C+I_D+I_E
\end{align*}
and with the help of \eqref{eq:phietaminussphietastrich}, the first of these integrals can be rewritten as
\begin{align*}
 I_A&=(N-1)(1-q)\inttaut\intnR r^{N-2}v_ru_r^2\left(u_r^2+\eta\right)^{\frac q2-1} + \eta(N-1)\inttaut\intnR r^{N-2}v_r\left(u_r^2+\eta\right)^{\frac q2-1}.
\end{align*}
Treating the second term similarly and inserting \eqref{eq:phietaminussphietastrich} and \eqref{eq:radial2} gives 
\begin{align*}
 I_B
 &=(1-2q) \inttaut\intnR r^{N-1} v_{rr}u_r^2\left(u_r^2+\eta\right)^{\frac q2-1} +\eta \inttaut\intnR r^{N-1} v_{rr} \left(u_r^2+\eta\right)^{\frac q2-1}\\
 &=(2q-1) \inttaut\intnR r^{N-1} (u-v)u_r^2\left(u_r^2+\eta\right)^{\frac q2-1} \\
&\;\;+(N-1)(2q-1) \inttaut\intnR r^{N-2} v_ru_r^2\left(u_r^2+\eta\right)^{\frac q2-1}+\eta \inttaut\intnR r^{N-1} \left(v-u-\frac{N-1}r v_r\right) \left(u_r^2+\eta\right)^{\frac q2-1},
\end{align*}
where also $(u-v)u_r^2\leq u\left(u_r^2+\eta\right)$. For the sum of these terms we are thereby led to 
\begin{align*}
 I_A+I_B&\leq
(N-1)q\inttaut\intnR r^{N-2}v_ru_r^2\left(u_r^2+\eta\right)^{\frac q2-1}+(2q-1) \inttaut\intnR r^{N-1} u\Phi_\eta(u_r)\\
&\;+\eta \inttaut\intnR r^{N-1} (v-u) \left(u_r^2+\eta\right)^{\frac q2-1},
\end{align*}
where we can use Lemma \ref{thm:radialestimate} to infer 
\begin{align*}
I_A+I_B &\leq q\frac{N-1}{N} \inttaut\intnR r^{N-2}r\norm[\Liom]{u} u_r^2\left(u_r^2+\eta\right)^{\frac q2-1}
\;\;+ (2q-1) \inttaut\intnR r^{N-1} u\Phi_\eta(u_r) \\
&\;\;+\eta \inttaut\intnR r^{N-1} (v-u) \left(u_r^2+\eta\right)^{\frac q2-1} \\
&\leq (q\frac{N-1}{N}+2q-1) \inttaut\norm[\Liom]{u} \intnR r^{N-1} \Phi_\eta(u_r) 
+\eta \inttaut\intnR r^{N-1} (v-u) \left(u_r^2+\eta\right)^{\frac q2-1}.
\end{align*}
Furthermore adding the other terms and making use of \eqref{eq:sphietastrich} in $I_D$, 

\begin{align*}
I_A+\ldots+I_E
&\leq\inttaut\left(((4-\frac1N)q-1)\norm[\Liom]{u}+\kappa q\right)\int_{\Om} \Phi_\eta\left(|\na u|\right)\\
&\quad+q\inttaut\norm[\Liom]u\int_{\Om} |\na v||\na u|(|\na u|^2+\eta)^{\frac q2-1}
+\eta \inttaut\intnR r^{N-1} (v-u) \left(u_r^2+\eta\right)^{\frac q2-1}\\
&\leq \inttaut\left(4q\norm[\Liom]{u}+\kappa q\right)\int_{\Om} \Phi_\eta\left(|\na u|\right)
+q \inttaut\norm[\Liom]{u}\int_{\Om} |\na v| (|\na u|^2+\eta)^{\frac {q-1}2}\\
&\qquad\qquad\qquad\qquad\qquad+\eta \inttaut\intnR r^{N-1} v \left(u_r^2+\eta\right)^{\frac q2-1}.
\end{align*}
Here an application of Young's inequality gives 
\begin{align*}
I_A+\ldots+I_E
&\leq \inttaut(4q\norm[\Liom]{u}+\kappa q)\int_{\Om} \Phi_\eta\left(|\na u|\right)
+\inttaut\norm[\Liom]{u} \int_{\Om} |\na v|^q \\
&\;\;+q\frac{q-1}{q} \inttaut\norm[\Liom] u\int_{\Om} \left(|\na u|^2+\eta\right)^{\frac {q}2}
+\eta  \inttaut\norm[\Liom]{v}\intnR r^{N-1} \left(u_r^2+\eta\right)^{\frac q2-1}.
\end{align*}
Merging first and third term, with Lemma \ref{thm:elliptregularity} (and $K$ as provided by that lemma) and Lemma \ref{thm:vlpleu} we have
\begin{align*}
I_A+\ldots+I_E
&\leq \inttaut(5q\norm[\Liom]{u}+\kappa q)\int_{\Om} \Phi_\eta\left(|\na u|\right)+K\inttaut\norm[\Liom]{u}\norm[\Lqom]{u}^q \\
&\;\;+\eta  \inttaut\norm[\Liom]{u}\intnR r^{N-1} \left(\frac{q-2}{q}\left(u_r^2+\eta\right)^{\frac {q-2}2\frac q{q-2}}+\frac 2q\cdot 1^{\frac{q}2}\right)\\
&\leq \inttaut(5q\norm[\Liom]{u}+\kappa q)\int_{\Om} \Phi_\eta\left(|\na u|\right)+\inttaut|\Om|\norm[\Liom]{u}^{1+q}\\
&\;\;+\eta \frac{q-2}{q}  \inttaut\norm[\Liom]{u}\int_{\Om} (|\na u|^2+\eta)^{\frac {q}2}        +\eta \inttaut \frac 2q  |\Om| \norm[\Liomtaut]{u}\\
&\leq \inttaut\left(\left(5q+\frac{q-2}q \eta\right)\norm[\Liom]{u}+\kappa q\right)\int_{\Om} \Phi_\eta\left(|\na u|\right)+|\Om|\inttaut\left(K\norm[\Liom]{u}^{1+q}+\frac {2\eta}q\norm[\Liomnt]{u}\right)
\end{align*}

In total, these estimates show the claim.
\end{proof}
\begin{remark}
 In the above proof (and all affected propositions), $5q$ could be replaced by $(5-\frac1N)q-2$.
\end{remark}

\begin{lemma}
\label{thm:estimateohneeta}
Under the assumptions of Lemma \ref{thm:estimate} the following holds:
 \[
\intom |\na u(\cdot,t)|^q\leq \left(\norm[\Lqom]{\na u(\cdot, \tau)}^q + \left(|\Om|K\intnt\norm[\Liom]{u(\cdot,s)}^{1+q}ds\right)\right)\exp\left( (5q\intnt\norm[\Liom]{u(\cdot,s)}ds +\kappa qt\right)
\]
\end{lemma}

\begin{proof}
Starting from Lemma \ref{thm:estimate}, by Gronwall's inequality we can conclude
\begin{align*}
 \int_{\Om}& \Phi_\eta(\na u(\cdot,t))\leq \left(\int_{\Om} \Phi_\eta(\na u(\cdot, \tau))+\inttaut|\Om|\norm[\Liom]{u}^{1+q}+\eta t \frac 2q  |\Om| \norm[\Liomnt]{u} \right)\\
&\qquad\qquad\qquad\qquad\qquad\qquad\qquad\qquad\qquad\qquad\qquad\qquad\cdot\exp\left( \inttaut\left(\left(5q+\frac{q-2}q \eta\right)\norm[\Liom]{u}+\kappa q\right) \right).
\end{align*}
By smoothness of $u$ in $\Ombar\times(0,T)$, we have for all $s\in(0,T)$
%As in the limit $\eta\downto 0$ by the dominated convergence theorem
\[
 \intom \Phi_\eta(\na u(\cdot,s)) \to \intom |\na u(\cdot,s)|^q \quad \textrm{ as }\eta\downto 0.
\]
%\Om bd, (.^2+\eta)^q/2\leq .^q+g(\eta), g cts, => (...)^q/2\leq .^q+1 (intbar), Lebesgue => Beh
From this we gain 
\[
 \int_{\Om} |\na u(\cdot,t)|^q \leq \left(\int_{\Om} |\na u(\cdot, \tau)|^q +|\Om| \inttaut\norm[\Liom]{u}^{1+q}\right)\exp\left( \inttaut(5q\norm[\Liom]{u}) +\kappa q t \right), 
\]
which implies the assertion.
\end{proof}
\begin{cor}
\label{thm:theestimate}
 In addition to the hypotheses of Lemma \ref{thm:estimate}, let $u_0$ be compatible. Then 
\[
 \int_{\Om} |\na u(\cdot,t)|^q \leq \left(\int_{\Om} |\na u_0|^q +|\Om| \intnt\norm[\Liom]{u}^{1+q}\right)\exp\left( \intnt(5q\norm[\Liom]{u}) +\kappa q t \right).
\]
\end{cor}
\begin{proof}
Since from the dominated convergence theorem we know that
\[
 \int_{\Om} |\na u(\cdot, \tau)|^q\to \intom|\na u_0|^q
\]
as $\tau\downto 0$ due to the boundedness of $\na u$ for solutions of \eqref{eq:epsprob} with compatible initial data, this is a direct consequence of Lemma \ref{thm:estimateohneeta}.
\end{proof}

\subsection{Epsilon-independent time of existence}
\label{sec:existencetime}
We begin this section with some Gronwall-type lemma which we will need during the next proof:
\begin{lemma}
\label{thm:owngronwall}
 Let $f\colon[0,\infty)\to \R$ nondecreasing and locally Lipschitz continuous, let $y_0\in\R$.
 Denote by $y$ the solution of $y(0)=y_0$, $y'(t)=f(y(t))$ on some interval $(0,T)$ and assume that the continuous function $z\colon[0,T)\to \R$ satisfies 
\[
 z(t)\leq z(0)+\intnt f(z(\tau)) d\tau\;\;\textrm{ for all }t\in(0,T),\qquad z(0)<y_0.
\]
Then $z(t)\leq y(t)$ for all $t\in(0,T)$.
\end{lemma}
\begin{proof}
 Let $T_0:=\inf\setl{t\in(0,T): z(t)>y(t)}$ and assume that $T_0<T$ exists. Due to continuity, $z(T_0)=y(T_0)$, i.e. 
\begin{align*}
 z(0)&+\intn{T_0} f(z(\tau))d\tau\geq z(T_0)=y(T_0)=y(0)+\intn{T_0}y'(\tau)d\tau\\
 &=y_0+\intn{T_0}f(y(\tau))d\tau
\geq y_0+\intn{T_0}f(z(\tau))d\tau
 >z(0)+\intn{T_0} f(z(\tau))d\tau,
\end{align*}
which is contradictory.
\end{proof}

The next lemma prepares the ground for the approximation procedure to be carried out in Theorem \ref{thm:existence}. It guarantees that solutions to \eqref{eq:epsprob} exist ``long enough''. Its proof is an adaption of that of \cite[Lemma 4.5]{winkler_14_ctexceed}, where an assertion similar to our Lemma \ref{thm:existence} is shown.

\begin{theorem}
 \label{thm:commonexistencetime}
 Let $\kappa\geq 0, \my>0, q>n$.
 Then for any $D>0$ there are some numbers $T(D)>0$ and $M(D)>0$ such that for any radially symmetric nonnegative and compatible $u_0\in\Weqom$ with $\norm[\Weqom]{u_0}\leq D$, for any $\eps>0$ the classical solution $(u_\eps,v_\eps)$ of \eqref{eq:epsprob} exists on $\Om\times(0,T(D))$ and $\norm[\Li(\Om\times(0,T(D))]{u_\eps}\leq M(D)$.
\end{theorem}
\begin{proof}
 For any $\eps>0$, the classical solution $(u_\eps,v_\eps)$ of \eqref{eq:epsprob} exists on some interval $(0,\Tmax^\eps)$ and satisfies $\limsup_{t\upto \Tmax^\eps} \norm[\Liom]{u_\eps}=\infty$, unless $\Tmax^\eps=\infty$.
 It is therefore sufficient to show boundedness of $\norm[\Liom]{u_\eps}$ on $(0,T(D))$ for some $\eps$-independent $T(D)>0$.
  Fix constants $c_1,c_2$ such that for all $\psi\in\Weqom$
\begin{equation}
 \label{eq:constantsforexistencetime}
 \norm[\Liom]\psi \leq c_1\norm[\Lqom]{\na\psi}+c_1\norm[\Leom]\psi \qquad \mbox{ and } \qquad \norm[\Leom]{\psi}\leq c_2\norm[\Weqom]{\psi},
\end{equation}
%\norm[Liom]{\psi}\leq\norm[\liom]{\psi-\psibar}+\sup \frac1{|\Om|} \norm[\Leom]{\psi}\\%q>n
%\leq \norm[\Weqom]{\psi-\psibar}+C\norm[\Leom]{\psi}\\
%\leq C \norm[\Lqom]{\psi-\psibar}+C\norm[\Lqom]{\na \psi-\underbrace{\na\psibar}_{=0}}+C\norm[\Leom]{\psi}\\%Poincare
%\leq C \norm{\Lqom}{\na\psi} +C\norm[\Lqom]{\na \psi}+C\norm[\Leom]{\psi}
% and 
% \[
%  \norm[\Leom]{\psi}\leq c_2\norm[\Weqom]{\psi}%\qquad \forall \psi\in\Weqom
% \]
where we use $q>n$, as well as $K$ as in Lemma \ref{thm:estimate} and $c_3=c_3(D)$ such that 
\[
 \frac{c_3}{c_1}=\max\set{c_2D,\frac{\kappa|\Om|}\my},
\]
so that by Corollary \ref{thm:ulpbounded} applied to $p=1$, 
\begin{equation}
\label{eq:integralexistencetime}
 \intom u\leq \frac{c_3}{c_1}. %\int u\leq \int u_0 max \frac \kappa\my |\Om| 
\end{equation}
Let furthermore denote $y_D$ the solution to 
\[
 y_D'(t)=(6q c_1+K|\Om|(2c_1)^{1+q})y_D^{1+\frac1q}+(6q c_3+\kappa q)y_D+|\Om|K(2c_3)^{1+q}+1), y_D(0)=(\sqrt{2}D)^q+1
\]
and denote by $T(D)>0$ a number, such that $y_D(t)\leq (\sqrt{2}D)^q+2$ for all $t\in(0,T(D))$.%ex because of loc. wellposedness and continuity.
 
 Now, let $u_0\in\Weqom$ be as specified in the lemma, especially with $\norm[\Weqom]{u_0}\leq D$. For $\eps>0$ denote by $u_\eps$ the solution of the corresponding equation \eqref{eq:epsprob}.

We apply Lemma \ref{thm:estimate} for conveniently small $\eta\in(0,\min\set{q,\frac12|\Om|^{-\frac2q},\frac{q}{2c_3}})$ and arbitrary $t\in(0,T(D))$ and obtain, in the limit $\tau\downto 0$ due to compatibility of $u_0$
\begin{align*}
  \int_{\Om} \Phi_\eta(|\na u_\eps(\cdot,t)|)\leq& \int_{\Om} \Phi_\eta(|\na u_0|)
+ \intnt\left(\left(6q \norm[\Liom]{u_\eps}+\kappa q\right)\int_{\Om} \Phi_\eta(|\na u_\eps|)\right)\\
&\qquad\qquad\qquad\qquad\qquad\qquad\qquad\qquad\qquad+|\Om|\intnt\left(K\norm[\Liom]{u_\eps}^{1+q}+\frac {2\eta}q \norm[\Liom]{u_\eps}\right)\\
 \leq& \intom \Phi_\eta(|\na u_0|)+\intnt \left((6q(c_1\norm[\Lqom]{\na u_\eps}+c_3)+\kappa q)\intom \Phi_\eta(|\na u_\eps|)\right)\\
&\qquad\qquad\qquad\qquad+|\Om|\intnt \left(K(c_1\norm[\Lqom]{\na u_\eps}+c_3)^{1+q}+\frac{2\eta}q(c_1\norm[\Lqom]{\na u_\eps}+c_3)\right)
\end{align*}
Here we abbreviate $\intom \Phi_\eta(\na u_\eps(\cdot,t)=:z(t)$ and estimate $\norm[\Lqom]{\na u_\eps(t)}\leq z^{\frac1q}(t)$. Then
\begin{align*}
 z(t)&\leq z(0)+\intnt \left((6q c_1+K|\Om|(2c_1)^{1+q})z^{1+\frac1q}(s)+(6q c_3+\kappa q)z(s)+ |\Om|(K(2c_3)^{1+q}+\frac {2\eta}{q}c_3)\right)ds.
\end{align*}
Additionally 
\begin{align*}
 z(0)=\intom \Phi_\eta(\na u_{0})&\leq 2^{\frac q 2}\intom |\na u_0|^q+(2\eta)^{\frac q2} |\Om|\leq (\sqrt{2}D)^q+1.
\end{align*}

Lemma \ref{thm:owngronwall} therefore leads us to the conclusion that, for all $t\in (0,T(D))$ and independent of $\eta$,  
\(
 \int \Phi_\eta(\na u_\eps(t))\leq y_D(t),
\)
which by Fatou's lemma implies 
\[\intom |\na u_\eps(t)|^q\leq 2D+2\]
 for all $t\in (0,T(D))$. 
Along with \eqref{eq:constantsforexistencetime} and \eqref{eq:integralexistencetime}, this shows that for all $\eps>0$ 
\[
 \norm[\Liom]{u_\eps(\cdot,t)}\leq c_1((\sqrt{2}D)^q+2)^\frac1q+c_3=:M(D)
\]
on $(0,T(D))$.
\end{proof}

\subsection{Preparations for convergence: boundedness of $u_t$ in an appropriate space}
\label{sec:utbd}

In order to use the Aubin-Lions-type estimate of Lemma \ref{thm:aubinlions}, we will need at least some regularity of the time derivative of bounded solutions. 
\begin{lemma}
 \label{thm:timederivative}
  Let $\eps_0>0$, let $\my>0, T>0, q>n, p\in(1,\infty), M>0$. 
 Let $u_{0\eps_0}\in\Weqom$ be compatible and nonnegative.
  Then there is $C>0$ such that the following holds for $\eps\in(0,\eps_0)$:
  If a solution $u_\eps$ of \eqref{eq:epsprob} in $Q_T$ with compatible nonnegative radial initial data $u_{0\eps}\in\Weqom$, $\norm[\Weqom]{u_{0\eps}-u_{0\eps_0}}$, satisfies 
  \[
   |u_\eps(x,t)|<M
  \]
  for all $(x,t)\in\Om\times[0,T)$, then 
  \[
   \norm[\LpnTX{(W^{1,\frac q{q-1}}(\Om))^*}]{u_{\eps t}}\leq C.
  \]
\end{lemma}
\begin{proof}
 Let $\psi\in \Ceombar$. Multiply \eqref{eq:epsprob} by $\psi$ and integrate over $\Om$:
 \begin{align*}
   \left|\intom u_{\eps t}\psi\right|&=\left|\intom \eps \Lap u_\eps \psi -\na(u_\eps\na v_\eps)\psi + \kappa u_\eps\psi -\my u_\eps^2\psi\right|\\
   &\leq \eps\left|\intom \na u \na \psi\right|+\left|\intom u_\eps \na v\eps\na \psi\right|+\kappa\norm[\Liom]{u_\eps}\left|\int \psi\right| +\my \norm[\Liom]{u_\eps}^2\left|\int \psi\right|.
 \end{align*}
 Invoking Corollary \ref{thm:theestimate}, Lemma \ref{thm:elliptregularity} and H\"older's inequality, we infer the existence of constants with $\norm[\Lqom]{\na u}\leq\Mtilde$, $\norm[\Lqom]{v}\leq \Ctilde\norm[\infty]{u}\leq \Ctilde M$ and $\norm[\Leom]{\psi}\leq\Chat\norm[L^\frac{q}{q-1}]{\psi}$ and conclude for $C:=T^{\frac1p}\max\set{\eps_0\Mtilde+\Ctilde M^2, (\kappa M+\my M^2)\Chat}$
 \begin{align*}
  \left|\intom u_{\eps t}\psi\right|&\leq \eps\norm[\Lqom]{\na u}\norm[L^{\frac q{q-1}}(\Om)]{\na \psi}+M\norm[\Lqom]{\na v_\eps}\norm[L^\frac{q}{q-1}]{\na \psi} +\kappa M \norm[\Leom]{\psi}+\my M^2\norm[\Leom]{\psi}\\
  &\leq 
  \left(\eps \Mtilde + \Ctilde M^2\right)\norm[L^\frac{q}{q-1}]{\na \psi}+(\kappa M+\my M^2)\Chat \norm[L^{\frac q{q-1}}]{\psi}\\
  &\leq C T^{-\frac1p} \norm[ W^{1,\frac q{q-1}}(\Om)]{\psi}.
 \end{align*}
Here taking the supremum over $\psi$ with $\norm[W^{1,\frac{q}{q-1}}(\Om)]{\psi}=1$ reveals 
\[
 \norm[(W^{1,\frac{q}{q-1}}(\Om))^*]{u_{\eps t}} \leq C T^{-\frac1p}
\]
and hence 
\[
 \left(\intnT\norm[(W^{1,\frac{q}{q-1}}(\Om))^*]{u_{\eps t}}^p\right)^\frac1p \leq C.\qedhere
\]
\end{proof}

%Theestimate, zum Abgleichen:
% \int_{\Om} |\na u(\cdot,t)|^q \leq \left(\int_{\Om} |\na u_0|^q +|\Om| \intnt\supnorm{u}^{1+q}\right)\exp\left( \intnt(5q\supnorm{u}) +\kappa q t \right).

\section{Hyperbolic-elliptic case}
\label{sec:limprob}
\subsection{What is a solution?}
We want to name a function ``solution'' if it is a solution in a sense similar to that in \cite[Def. 4.1]{winkler_14_ctexceed}:
\begin{defn}\label{def:soln}Let $T\in(0,\infty]$. A strong \weqsoln\ of \eqref{eq:limprob} in $Q_T=\Om\times(0,T)$ is a pair of functions $u\in C(\Ombar\times[0,T))\cap\LilocnTWeqom$ and $v\in C^{2,0}(\Ombar\times[0,T))$ such that $u,v$ nonnegative, $v$ classically solves $0=\Lap v-v+u$ in $\Om$, $\delny v\amrand=0$ and 
\begin{equation}
 \label{eq:defsoln}
 -\intnT\intom u\phii_t - \intom u_0\phii(\cdot,0) = \intnT\intom u\na v \na \phii+\kappa\intnT\intom u\phii-\my\intnT\intom u^2\phii
\end{equation}
 holds true for all $\phii\in L^1((0,T);\Weeom)$ that have compact support in $\Ombar\times [0,T)$ and satisfy $\phii_t\in\LeomnT$.
 If additionally $T=\infty$, we call the solution global.
\end{defn}
\begin{remark}
 Due to density arguments, it is of course possible to formulate Definition \ref{def:soln} for $\phii\in\Cninf(\Ombar\times [0,T))$ and obtain the same solutions.
\end{remark}

\subsection{Uniqueness}
These solutions are unique as can be proven very similar to the one-dimensional case.
\begin{lemma}
 \label{thm:unique}
 Let $q>n$, $T\in(0,\infty]$ and $u_0\in\Weqom$ with $u_0\geq 0$. Then the \weqsoln\ of \eqref{eq:limprob} in $Q_T$ is unique.%if it exists
\end{lemma}

\begin{proof} (Cf. \cite[Lemma 4.2]{winkler_14_ctexceed}).
Let $q>n$ and $u_0\in\Weqom$ with $u_0\geq 0$. Let $(u,v),(\utilde,\vtilde)$ be strong \weqsoln s of \eqref{eq:limprob} and note that $(w,z):=(u-\utilde,v-\vtilde)$ satisfy
\begin{equation}
 \label{eq:solndefintegral}
 -\intnT\intom w\phii_t=\intnT\intom (u\na v-\utilde\na \vtilde)\na \phii+ \kappa\intnT\intom w\phii-\my\intnT\intom(u^2-\utilde^2)\phii
\end{equation}
for all $\phii\in L^1((0,T);\Weeom)$ which have compact support in $\Ombar\times [0,T)$ and satisfy $\phii_t\in\LeQT$ 
and 
\begin{equation}
 \label{eq:differenzellipt}
 0=\Lap z+w-z.
\end{equation}
Let $T_0\in(0,T)$. Then by Definition \ref{def:soln} (and by Lemma \ref{thm:elliptregularity}), we can define constants such that 
\begin{align*}
 c_1:=\norm[\LiomnTn]{v},\quad c_2:=\norm[\LiomnTn]{u},\quad c_3=\norm[\Li((0,T_0),\Lqom]{\utilde},\quad c_5=\norm[\LiomnTn]{\utilde}
\end{align*}
and let $c_4$ denote the constant from Lemma \ref{thm:elliptregularity}.
We set 
\begin{equation}
 \label{eq:theuniquenessconst}
C:=q(c_1+c_2+\kappa+c_3c_4+c_5).
\end{equation}

By Lemma \ref{thm:elliptregularity}, \eqref{eq:differenzellipt} implies $\norm[\Wzqom]{z}\leq C\norm[\Lqom]{w}$ and hence, as for $q>n$ $\Weqom\embeddedinto\Liom$, 
\[
 \norm[\Liom]{\na z}\leq C\norm[\Lqom]{w}.
\]

In \eqref{eq:solndefintegral} we use some function $\phii$ we construct as follows:
For $t_0\in(0,T_0)$ define $\chi_\delta\in\WeiR$ by 
\begin{equation*}
 \chi_\delta(t):=\begin{cases}1,&t<t_0,\\
		  \frac{t_0-t+\delta}{\delta},&t\in[t_0,t_0+\delta],\\
                  0,&t>t_0+\delta,
                 \end{cases}
\end{equation*}
for $\delta\in(0,\frac{T_0-t_0}{2})$ and let
\[
 \phii(x,t):=\chi_\delta(t)\frac1h\inttth w(x,s)(w^2(x,s)+\eta)^{\frac q2-1} ds.
\]
Then for $\delta\in(0,\frac{T_0-t_0}{2}), h\in(0,\frac{T_0-t_0}{2}), 1>\eta>0$, $\phii$ is a valid test function in \eqref{eq:solndefintegral} and yields:
\begin{align*}
  \frac1\delta&\inttntnd\intom w(x,t)\frac1h\inttth w(x,s)(w^2(x,s)+\eta)^{\frac q2-1} dsdxdt \\
&\qquad\qquad\qquad\qquad-\intnT\intom\chi_\delta(t)w(x,t)\frac{w(x,t+h)(w^2(x,t+h)+\eta)^{\frac q2-1}-w(x,t)(w^2(x,t)+\eta)^{\frac q2-1}}h\\
  &= \intnT\intom\chi_\delta(t)[-\na u\na v -u\Lap v+\na\utilde\na \vtilde+\utilde\Lap \vtilde+\kappa w-\my w(u+\utilde)]
\cdot \frac1h \inttth w(x,s)(w^2(x,s)+\eta)^{\frac q2-1}\\
 &=\intnT\intom \chi_\delta(T)[-\na w\na v-w\Lap v - \na \utilde\na z-\utilde\Lap z+\kappa w-\my w(u+\utilde)]
\cdot\frac1h \inttth w(x,s)(w^2(x,s)+\eta)^{\frac q2-1}. 
\end{align*}
Taking the limit $\delta\downto 0$, which is possible for the first term because $(x,t)\mapsto w(x,t)\frac1h \inttth w(x,s)(w^2(x,s)+\eta)^{\frac q2-1}ds $ is continuous and on the right hand side by Lebesgue's theorem, since $\na z,\Lap z, \utilde, \na v, \Lap v, w, u,$ are bounded and $\na w$ is uniformly bounded in $\Lqom$ up to time $t_0+h$ (according to Definition \ref{def:soln}), we obtain 
\begin{align}
\label{eq:uqgetestet}
  &\intom w(x,t_0)\frac1h\inttntnh w(x,s)(w^2(x,s)+\eta)^{\frac q2-1} ds  dx \notag\\
&\qquad\quad\qquad\qquad-\intntn \intom w(x,t) \frac{w(x,t+h)(w^2(x,t+h)+\eta)^{\frac q2-1}-w(x,t)(w^2(x,t)+\eta)^{\frac q2-1}}h dxdt\\
 &=\intntn\intom(-\na w\na v -w\Lap v-\na \utilde\na z-\utilde\Lap z+\kappa w-\my w(u+\utilde)]
\cdot \frac1h\inttth w(x,s)(w^2(x,s)+\eta)^{\frac q2-1}dsdxdt.\notag
\end{align}
With the abbreviations $w=w(x,t)$, $w_h=w(x,t+h)$ we observe that 
\begin{align*}
 -\frac1h\intntn&\intom ww_h(w_h^2+\eta)^{\frac q2-1}+\frac1h\intntn\intom w^2(w^2+\eta)^{\frac q2-1}\\
 &\geq -\frac1h\intntn\intom(w^2+\eta)^{\frac 12}(w_h^2+\eta)^{\frac 1 2}(w_h^2+\eta)^{\frac q2-1}+\frac1h\intntn\intom w^2(w^2+\eta)^{\frac q2-1}\\
 &\geq -\frac1h\frac1q\intntn\intom(w^2+\eta)^{\frac q2}-\frac1h \frac{q-1}q\intntn\intom(w_h^2+\eta)^{\frac q2}\quad+\frac1h\intntn\intom w^2(w^2+\eta)^{\frac q2-1},
\end{align*}
where we have used that $s\leq (s^2+\eta)^\frac12$ and Young's inequality. Converting the time shift in the arguments to a change of integration limits, we obtain
\begin{align*}
 -&\frac1h \frac1q\intntn\intom w^2(v,t)(w^2+\eta)^{\frac q2-1}-\frac1h\eta\frac1q\intntn\intom(w^2+\eta)^{\frac q2-1}
 -\frac1h\frac{q-1}q\int_h^{t_0+h}\intom w^2(w^2+\eta)^{\frac q2-1}\\
&\qquad\qquad\qquad\qquad\qquad\qquad\qquad\qquad -\frac1h\frac{q-1}q\eta\int_h^{t_0+h}\intom(w^2+\eta)^{\frac q2-1}
 +\frac1h\intntn\intom w^2(w^2+\eta)^{\frac q2-1}\\
 &=\frac{q-1}{q}(\frac1h\intnh\intom w^2(w^2+\eta)^{\frac q2-1}-\frac1h\inttntnh\intom w^2(w^2+\eta)^{\frac q2-1})\\
 &\qquad\qquad\qquad\qquad\qquad\qquad\qquad\qquad-\eta\left(\frac1{hq}\intntn\intom(w^2+\eta)^{\frac q2-1}+\frac{q-1}{qh}\inthtnh\intom(w^2+\eta)^{\frac q2-1}\right),
\end{align*}
where in the limit $\eta\downto 0$ the last line vanishes. %by Lebesgue
Furthermore, by the continuity of $w$ and because $|w(w^2+\eta)^{\frac q2-1}|$ can be bounded by the integrable function $(w^2+1)^{\frac q2}$ 
\[
 \setl{\Om\ni x\mapsto \frac1h\inttntnh w(x,s)(w^2(x,s)+\eta)^{\frac q2-1} ds}\weakstarto \setl{x\mapsto \frac1h\inttntnh w(x,s)|w(x,s)|^{q-2}ds} \text { in }\Liom
\]
as well as 
\begin{align*}
 \setl{\Om\times(0,t_0)\ni (x,t)\mapsto \frac1h\inttth w(x,s)(w^2(x,s)+\eta)^{\frac q2-1}ds}
\weakstarto \setl{(x,t)\mapsto \frac1h\inttth w(x,s)|w(x,s)|^{q-2}ds }
\end{align*}
in $\Liomntn$ as $\eta\downto 0$.
Therefore we can conclude from \eqref{eq:uqgetestet} that 
\begin{align*}
 \intom w(x,t_0)\frac1h\inttntnh w(x,s)|w(x,s)|^{q-2} dsdx+\frac{q-1}{qh}\intnh\intom|w(x,t)|^q-\frac{q-1}{qh}\inttntnh\intom|w(x,t)|^qdxdt\\
\leq \intntn\intom(-\na w\na v -w\delta v-\na\utilde\na z-\utilde\Lap z+\kappa w-\my w(u+\utilde))\frac1h\inttth w(x,s)|w(x,s)|^{q-2}
\end{align*}
for $h\in(0,\frac{T_0-t_0}2)$. 
As $w$ is continuous on $\Ombar\times[0,T_0]$ and $w(\cdot,0)=0$ in $\Om$, $h\downto 0$ yields
\[
 \frac1q\intom|w(x,t_0)|^q\leq\intntn\intom(-\na w\na v -w\Delta v-\na\utilde\na z-\utilde\Lap z+\kappa w-\my w(u+\utilde)) w(x,t)|w(x,t)|^{q-2}.
\]
Here we will estimate the integral on the right hand side to obtain an expression that allows to conclude $w=0$ by means of Gronwall's lemma. 
We will consider the summands seperately:
\[
 -\intntn\intom\na w\na v w|w|^{q-2}=\frac1q \intntn\intom |w|^q \Lap v=\frac1q\intntn\intom|w|^q(v-u)\leq\norm[\LiomnTn]{v}\intntn\intom|w|^q
\]
Also for the next term, the second equation of \eqref{eq:epsprob} and nonnegativity of $v$ are helpful:
\[
 -\intntn\intom w\Lap v w|w|^{q-2}=-\intntn \intom |w|^q(v-u)\leq \norm[\LiomnTn]{u}\intntn\intom |w|^q.
\]
For the last we make use of the nonnegativity of both $u$ and $\utilde$:
\[
 \intntn\intom(\kappa w-\my w(u+\utilde))w|w|^{q-2}=\intntn\intom (\kappa-\my(u+\utilde))|w|^q\leq \kappa \intntn\intom|w|^q.
\]
Boundedness of $\utilde$ in $\Weqom$ and Lemma \ref{thm:elliptregularity} play the main role in the following estimate.
\[
 -\intntn\intom \na\utilde\na z w|w|^{q-2}\leq \intntn\norm[\Lqom]{\na\utilde}\norm[\Liom]{\na z}\norm[L^{\frac q{q-1}}(\Om)]{w|w|^{q-2}}
 \leq \intntn c_3c_4\norm[\Lqom]{w}\norm[\Lqom]{w}^{q-1}.
\]
And finally, once more employing the second equation,
\[
 -\intntn\intom \utilde\Lap z w|w|^{q-2}=-\intntn\intom \utilde zw|w|^{q-2}+\intntn\intom \utilde |w|^q\leq
 c_5\intntn\intom|w|^q.
\] 
Gathering all these estimates together we gain the constant $C$ from \eqref{eq:theuniquenessconst} such that for all $t_0\in(0,T_0)$
\[
  \intom|w(x,t_0)|^q dx\leq C\intntn\intom|w|^q,
\]
hence by Gronwall's lemma $w=0$ (and therefore also $z=0$), which proves uniqueness of solutions.
\end{proof}

\subsection{Local existence, approximation}
%TODO Verweise in den Beweis schieben
We will prove existence of solutions to \eqref{eq:limprob} by means of a compactness argument whose key lies in: 
\begin{lemma}
 \label{thm:aubinlions}
 Let $X,Y,Z$ be Banach spaces such that $X\embeddedinto Y\embeddedinto Z$, where the embedding $X\embeddedinto Y$ is compact. Then for any $T>0, p\in (1,\infty]$, the space 
\[
 \setl{w\in\LiNTX X;w_t\in\LpnTX Z}
\]
is compactly embedded into $C([0,T];Y)$.
\end{lemma}
\begin{proof}
 The proof uses the Arzel\'a-Ascoli theorem and Ehrling's lemma and can be found in \cite[Lemma 4.4]{winkler_14_ctexceed}.
\end{proof}

We directly take this tool to its use and employ it with a slightly different choice of spaces than in the one-dimensional case to obtain a similar result.
\begin{lemma}
(Cf. \cite[Lemma 4.3]{winkler_14_ctexceed})
\label{thm:kgztosolnifbd}
 Let $\kappa\geq 0, \my >0, q>n$, assume $u_0\in\Weqom$ nonnegative and radially symmetric.
  Suppose that $(\eps_j)_{j\nat}\subset (0,\infty)$, $(u_{0\eps_j})_j\subset\Weqom$, $T>0$, $M>0$ are such that $\eps_j\downto 0$ as $j\to \infty$, $u_{0\eps}$ compatible and radial %and nonneg
  and such that whenever $\eps\in(\eps_j)_{j\nat}$, for the solution $(u_\eps,v_\eps)$ of \eqref{eq:epsprob} with initial condition $u_{0\eps}$, we have 
  \begin{equation}
   \label{eq:bounded}
   u_\eps(x,t)\leq M.
  \end{equation}
 for all $(x,t)\in Q_T$. Then there exists a strong \weqsoln\ $(u,v)$ of \eqref{eq:limprob} in $Q_T$ such that 
 \begin{align*}
  u_\eps\to u \qquad&\textrm{ in } C(\Qbar_T),\\
  u_\eps\wstarto u \qquad &\textrm{ in } L^\infty((0,T),\Weqom),\\
  v_\eps\to v \qquad &\textrm{ in }C^{2,0}(\Qbar_T). 
 \end{align*}
\end{lemma}

\begin{proof}
 According to Corollary \ref{thm:theestimate} and by \eqref{eq:bounded},
 \begin{equation*}
  \label{eq:bdinliweq}
  (u_{\eps_j})_j \text{is bounded in } L^\infty((0,T),\Weqom).
 \end{equation*}
 Lemma \ref{thm:timederivative} gives boundedness of the time derivatives: For some $p>1$, 
 \[
  (u_{\eps_j t})_j \text{ is bounded in } L^{p}((0,T),(W^{1,\frac{q}{q-1}}(\Om))^*). 
 \]
With the choice of $X=\Weqom, Y=C^\alpha(\Ombar), Z=(W^{1,\frac{q}{q-1}}(\Om))^*$, Lemma \ref{thm:aubinlions} allows to conclude relative compactness of $(u_{\eps_j})_j$ in $C([0,T],C^\alpha(\Om))$.
 Due to this and \eqref{eq:bdinliweq}, given any subsequence of $(\eps_j)_j$, we can pick a further subsequence thereof such that 
 \begin{align}
  \label{eq:limone}
   u_{\eps_{j_i}}\to u \qquad &\textrm{ in }C([0,T],C^\alpha(\Om)),\\
  \label{eq:limtwo}
   u_{\eps_{j_i}}\wstarto u\qquad &\textrm{ in }  L^\infty((0,T),\Weqom)
 \end{align}
 as $i\to \infty$ and also by the propagation of the Cauchy-property from $(u_{\eps_{j_i}})_{i\in\N}$ to $(v_{\eps_{j_i}})_{i\nat}$ via
\begin{align*}
 \norm[C^{2,0}(Q_T)]{v_{\eps_{j_i}}-v_{\eps_{j_k}}}&\leq\norm[C^{2+\alpha,0}(Q_T)]{v_{\eps_{j_i}}-v_{\eps_{j_k}}}\leq C\norm[C^{\alpha,0}(Q_T)]{u_{\eps_{j_i}}-u_{\eps_{j_k}}} = C\norm[{C([0,T],C^\alpha(\Om))}]{u_{\eps_{j_i}}-u_{\eps_{j_k}}}
\end{align*}
for all $i,k\nat$, where $C$ is the constant from Lemma \ref{thm:elliptregularity}, 
 \begin{align}
  \label{eq:limthree}
  v_{\eps_{j_i}}\to v \qquad C^{2,0}(Q_T).
 \end{align}
 The limit $(u,v)$ is a strong \weqsoln\ of \eqref{eq:limprob}, as can be seen by testing \eqref{eq:epsprob} by an arbitrary $\phii\in\Cninf(\Ombar\times[0,T))$ and taking $\eps=\eps_{j_i}\to 0$ in each of the integrals seperately, as possible by \eqref{eq:limone} to \eqref{eq:limthree}:
 \[
  -\intnT\intom u \phii_t-\intom u_{0}\phii(\cdot,t)=-0\cdot\intnT\intom\na u\na \phii-\intnT\intom u\na v \na \phii+\kappa \intnT\intom u \phii-\my \intnT\intom u^2\phii.
 \]
 Hence the limit of all these subsequences $u_{\eps_{j_i}}$ of subsequences is the same, namely the unique (Lemma \ref{thm:unique}) solution of \eqref{eq:limprob} and therefore the whole sequence converges (in the spaces indicated in equations \eqref{eq:limone} to \eqref{eq:limthree}) to the solution $(u,v)$ of \eqref{eq:limprob}. 
\end{proof}

In Lemma \ref{thm:kgztosolnifbd} we assumed uniform boundedness of the approximating solutions. Fortunately, on small time scales we are entitled to do so and can prove the following:  
\begin{lemma}%compare \cite[Lemma 4.5]{winkler_14_ctexceed}
\label{thm:shorttimeexistencelimprob}
 Let $\kappa\geq 0, \my>0, q>n$. Then for $D>0$ there is some $T(D)>0$ such that for any radial symmetric nonnegative $u_0\in \Weqom$ fulfilling $\norm[\Weqom]{u_0}<D$ there is a unique $\Weqom$-solution $(u,v)$ of \eqref{eq:limprob} in $\Om\times (0,T(D))$.
 Furthermore, if $u_{0\eps}$ are compatible functions satisfying $\norm[\Weqom]{u_{0\eps}-u_0}<\eps$, this solution $(u,v)$ can be approximated by solutions $(u_\eps,v_\eps)$ of \eqref{eq:epsprob} (with initial condition $u_{0\eps}$) in the following sense:
 \begin{align}
   u_\eps\to u \qquad&\textrm{ in } C(\Ombar\times[0,T])\\
   u_\eps\wstarto u \qquad &\textrm{ in } L^\infty((0,T),\Weqom)\\
   v_\eps\to v \qquad &\textrm{ in }C^{2,0}(\Om\times(0,T)). 
 \end{align}
Moreover, with $K$ as in Lemma \ref{thm:estimate} this solution satisfies
 \[
  \int_{\Om} |\na u(\cdot,t)|^q \leq \left(\int_{\Om} |\na u_0|^q +K|\Om| \intnt\norm[\Liom]{u}^{1+q}\right)\exp\left( \intnt(5q\norm[\Liom]{u}) +\kappa q t \right)
 \]
for a.e. $t\in(0,T(D))$.
\end{lemma}

\begin{proof}
For $\eps\in(0,1)$ let $u_{0\eps}$ be compatible and $\norm[\Weqom]{u_{0\eps}-u_0}<\eps$. Apply Theorem \ref{thm:commonexistencetime} with $D+1$ to obtain $T(D)$ such that the solutions $u_\eps$ to \eqref{eq:epsprob} with initial data $u_0$ exist on $(0,T(D))$ and are bounded by $M(D)$ on that interval.
Here Lemma \ref{thm:kgztosolnifbd} applies to provide a strong \weqsoln\ with the claimed approximation properties. The inequality results from Corollary \ref{thm:theestimate} as follows:
\\According to Corollary \ref{thm:theestimate}, for all $t\in(0,T(D))$, 
\[
 \int_{\Om} |\na u_\eps(\cdot,t)|^q \leq \left(\int_{\Om} |\na u_{0\eps}|^q +K|\Om| \intnt\norm[\Liom]{u_\eps}^{1+q}\right)\exp\left( \intnt(5q\norm[\Liom]{u_\eps}) +\kappa q t \right).
\]
Let $t\geq 0$ be such that $u(\cdot,t)\in \Weqom$. Convergence of the right hand side is obvious because of the uniform convergence $u_\eps\to u$ and $u_{0\eps}\to u_0$ in $\Weqom$.
This implies boundedness of $(\na u_{\eps_j})_j$ in $L^q$, hence $L^q$-weak convergence along a subsequence and - due to the  weak lower semicontinuity of the norm - %Dass der zugeh. Grenzwert die Ableitung ist, folgt aus schwacher kgz zusammen mit der glm konv. von u(\cdot,t).
\[
 \intom|\na u(\cdot,t)|^q \leq \liminf_{k\to \infty} \intom|\na u_{\eps_{j_k}}(\cdot,t)|^q  \leq \left(\intom|\na u_{0}|^q+|\Om|\intnt\norm[\Liom]{u}^{1+q}\right)\exp\left(\int 5q \norm[\Liom]{u}+\kappa q t\right).
\]
\end{proof}

\subsection{Continuation and existence on maximal time intervals}
Solutions constructed up to now may only exist on very short time intervals. With the following theorem (which parallels \cite[Thm. 1.2]{winkler_14_ctexceed} in statement and proof) we ensure that they can be glued together to yield a solution on a maximal time interval -- to all eternity or until blow-up.
\begin{theorem}
\label{thm:existence}
 Let $\kappa\geq 0, \my>0$, for some $q>n$ suppose $u_0\in\Weqom$ is nonnegative and radially symmetric. Then there exist $\Tmax\in(0,\infty]$ and a unique pair $(u,v)$ of functions
\begin{align*}
 u&\in\CombarNTmax\cap\LilocNTmaxWeqom,\\
 v&\in\CzombarNTmax
\end{align*}
that form a strong \weqsoln\  of \eqref{eq:limprob} and which are such that 
\begin{equation}
\label{eq:ext}
\textrm{either }\quad\Tmax=\infty \quad\textrm{ or }\quad \limsuptTmax \norm[\Liom]{u(\cdot, t)}=\infty
\end{equation}
\end{theorem}

\begin{proof}
 Apply Lemma \ref{thm:shorttimeexistencelimprob} to $D:=\norm[\Weqom]{u_0}$ to gain $T>0$ and a strong \weqsoln\ $(u,v)$ of \eqref{eq:limprob} in $\Om\times(0,T)$ fulfilling 
\begin{equation}
 \label{eq:l45star} 
 \intom|\na u|^q\leq \left(\intom|\na u_0|^q +K|\Om|\intnt \norm[\Liom]{u}^{1+q}\right)\exp(5q\intnt\norm[\Liom]{u}+\kappa q t)
\end{equation}
for almost every $t\in(0,T)$.
Accordingly, the set 
\begin{align*}
 S:=&\set{\Ttilde>0\; |\; \exists \text{ strong \weqsoln\ of \eqref{eq:limprob} in }\Om\times(0,\Ttilde)\\&\qquad\qquad \text{with initial condition }u_0\text{ and satisfying \eqref{eq:l45star} for a.e. }t\in(0,\Ttilde)}
\end{align*}
is not empty and $\Tmax:=\sup S\leq\infty$ is well-defined.

According to Lemma \ref{thm:unique} the strong \weqsoln\ on $\Om\times (0,\Tmax)$, which obviously exists, is unique. We only have to show the extensibility criterion \eqref{eq:ext}.

Assume $\Tmax<\infty$ and $\limsuptTmax\norm[\Liom]{u(\cdot,t)}<\infty$.
Then there exists $M>0$ such that for all $(x,t)\in\Om\times(0,\Tmax)$
\[
 u(x,t)\leq M.
\]
Let $N\subset (0,\Tmax)$ be a set of measure zero, as provided by the definition of $S$, such that \eqref{eq:l45star} holds 
for all $t\in(0,\Tmax)\setminus N$.

Together with $u\leq M$ %for the non-gradient part of the norm.
this would imply 
\[
 \norm[\Weqom]{u(\cdot,t_0)}\leq D_1 
\]
for some positive $D_1$ and for each $t_0\in(0,\Tmax)\ohne N$. Lemma \ref{thm:shorttimeexistencelimprob} would yield the existence of a strong \weqsoln\ of 
\begin{align*}
 \uhat_t&=-\na(\uhat\na \vhat)+\kappa\uhat-\my\uhat^2\\
0&=\Lap\vhat-\vhat+\uhat\\
0&=\delny\uhat\amrand=\delny \vhat\amrand,\\
\uhat(x,0)&=u(x,t_0)
\end{align*}
on $\Om\times(0,T(D_1))$, which would satisfy
\begin{equation}
\label{eq:lateestimate}
 \intom|\na\uhat(t)|^q\leq \left(\intom |\na u(t_0)|^q+K|\Om|\intnt\norm[\Liom]{u}^{1+q}\right)\exp\left(5q\intnt\norm[\Liom]{u}+\kappa q t\right)
\end{equation}
for almost every $t\in(0,T(D_1))$.\\
Upon the choice of $t_0\in(0,\Tmax)\ohne N$ with $t_0>\Tmax-\frac{T(D_1)}2$, 
\[
 (\utilde,\vtilde)(\cdot,t)=\begin{cases}(u,v)(\cdot,t)&t\in(0,t_0)\\(\uhat,\vhat)(\cdot,t-t_0)&t\in[t_0,t_0+T(D_1))\end{cases}
\]
would define a strong \weqsoln\ of \eqref{eq:limprob} in $\Om\times(0,t_0+T(D_1))$ which clearly would satisfy \eqref{eq:l45star} for a.e. $t<t_0$.
For $t>t_0$ on the other hand, a combination of \eqref{eq:lateestimate} and \eqref{eq:l45star} would give 
\begin{align*}
 \intom|\na \utilde(t)|^q
&\leq \left(\left(\intom|\na u_0|^q+K|\Om|\intntn\norm[\Liom]{u}^{1+q}\right)\exp\left(5q\intntn\norm[\Liom]{u}+\kappa q t_0\right)+K|\Om|\int_{t_0}^t\norm[\Liom]{u}^{1+q}\right)\\
&\qquad\qquad\qquad\qquad\qquad\qquad\qquad\qquad\qquad\qquad\qquad\qquad\qquad \cdot\exp(5q\int_{t_0}^t\norm[\Liom]{u}+\kappa q(t-t_0))\\
&\leq\left(\intom|\na u_0|^q+K|\Om|\intnt \norm[\Liom]{u}^{1+q}\right)\exp(5q\intnt\norm[\Liom]{u}+\kappa q t)
\end{align*}
This would finally lead to a contratiction to the definition of $\Tmax$ as supremum, because then obviously $(\utilde,\vtilde)$ would satisfy \eqref{eq:l45star} for a.e. $t\in(0,t_0+T(D_1))$.
\end{proof}

\subsection{An estimate for strong solutions: boundedness in $L^1$}
Having confirmed existence and uniqueness of solutions, we set out to explore some more of their properties. And as in \cite[Lemma 4.1]{winkler_14_ctexceed}, one of the first facts that can be observed (and proven like in the 1-dimensional case) is their boundedness in $L^1$.
\begin{lemma}
 \label{thm:limproblebd}%q>n f\"ur die Einbettung zu Beginn des Beweises.
 Let $\kappa\geq 0,\my>0$, assume that for $T>0,q> n$, $u$ is a strong \weqsoln\ of \eqref{eq:limprob} in $Q_T$ with $u_0\in\Weqom, u_0$ nonnegative. Then for all $t\in(0,T)$
\[
 \intom u(x,t)dx \leq \max\setl{\intom u_0,\frac{\kappa|\Om|}\my}.
\]
\end{lemma}
\begin{proof}
 Define $y(t)=\intom u(x,t)dx.$ Then $y$ is continuous, as 
$u\in C(\Ombar\times[0,T))$ and $u\in L^\infty_{loc}([0,T),\Weqom)\eingebettetin L^\infty_{loc}([0,T),L^1(\Om))\eingebettetin L^1_{loc}([0,T)\times\Ombar)$, 
and it is sufficient to show that $y\in C^1((0,T))$ and for all $t\in(0,T)$
\[
 y'(t)\leq \kappa y(t)-\frac\my{|\Om|} y^2(t).
\]
To see this, let $t_0\in(0,T), t_1\in(t_0,T)$ and let $\chi_\delta\in\WeiR$ be given by 
\[
 \chi_\delta(t)=\begin{cases}
                 0&t<t_0-\delta \vee t>t_1+\delta\\
		  \frac{t-t_0+\delta}{\delta}&t\in[t_0-\delta,t_0]\\
		  1&t\in(t_0,t_1)\\
		  \frac{t_1-t+\delta}{\delta}&t\in[t_1,t_1+\delta]
                \end{cases}
\]
for $\delta\in(0,\delta_0)$ with $\delta_0=\min\set{t_0,T-t_1}$. Then $\phii(x,t):=\chi_\delta(t)$, $(x,t)\in \Om\times(0,T)$ defines an admissible testfunction in \eqref{eq:defsoln} and we have 
\[
 -\intnT\intom u\phii_t-\intom u_0 \phii(\cdot,0) = \intnT\intom u\na v \na \phii+\kappa\intnT\intom u\phii-\my\intnT\intom u^2\phii
\]
so that, by $\chi_\delta(0)=0$ and $\na \phii=0$, 
\[
 -\frac1\delta\int_{t_0-\delta}^{t_0} \intom u+\frac1\delta\int_{t_1}^{t_1+\delta}\intom u=\kappa\intnT\intom \chi_\delta(t)u(x,t)dxdt-\my\intnT\intom u^2\chi_\delta(t)dxdt.
\]
Since $u$ is continuous, the left hand side converges to $y(t_1)-y(t_0)$, whereas the right hand side makes application of the dominated convergence theorem possible due to the boundedness of $u$ on $[0,T-\delta_0]$ %($u\in C(\Ombar\times [0,T-\delta_0])$, hence bounded and also $u^2$ is a suitable majorant) 
as $\delta\downto 0$ and we arrive at
\begin{align*}
 y(t_1)-y(t_0)&=\kappa\int_{t_0}^{t_1} \intom u-\my\int_{t_0}^{t_1}\intom u^2 dxdt\leq \kappa\int_{t_0}^{t_1} u-\my \int_{t_0}^{t_1} \frac 1{|\Om|}\left(\int u\right)^2.
\end{align*}
Upon division by $t_1-t_0$ and taking limits $t_1\to t_0$, we infer that indeed $y\in C^1((0,T))$ with 
\[
 y'(t)\leq \kappa y(t)-\frac \my {|\Om|}y^2(t)\qquad \textrm{ for all } t\in(0,T).\qedhere
\]
\end{proof}
\subsection{Global existence for large $\my$}
Bounds on the $\Liom$-norm are the only thing we need to guarantee existence of solutions for longer times. They arise as a corollary to Lemma \ref{thm:epsprobbounds}, which directly implies the following.

\begin{corollary}
\label{thm:limprobbounds}
 Let $\kappa\geq 0,\my\geq 1, q>n$. For each nonnegative, radial $u_0\in\Weqom$, \eqref{eq:limprob} has a unique global strong \weqsoln\ $(u,v)$. Furthermore, if $u_0\not\equiv 0$, then 
\[
 \norm[\Liom]{u(\cdot,t)}\leq\begin{cases}
  \frac{\kappa}{\my-1}(1+(\frac{\kappa}{(\my-1)\norm[\Liom]{u_0}}-1)e^{-\kappa t})\inv,&\kappa>0,\my>1,\\
  \frac{\norm[\Liom]{u_0}}{1+(\my-1)\norm[\Liom]{u_0}t},&\kappa=0,\my>1,\\
  \norm[\Liom]{u_0}e^{\kappa t},&\kappa>0, \my=1,\\
  \norm[\Liom]{u_0},&\kappa=0, \my=1.
  \end{cases}
\]
\end{corollary}
\begin{proof}
Local existence up to a maximal time $\Tmax\leq\infty$ is given by Theorem \ref{thm:existence}. For each $T\in(0,\Tmax)$, there are solutions of \eqref{eq:epsprob} converging to $(u,v)$ in $C(\Ombar\times(0,T))$, hence $u$ inherits the bounds from Lemma \ref{thm:epsprobbounds}. By \eqref{eq:ext}, $(u,v)$ must thus be global.
\end{proof}
Without further labour, we can state what we have obtained so far: %as analogue of \cite[Thm 1.3]{winkler_14_ctexceed}.
\begin{proposition}
\label{thm:limprobglobal}
 Let $\kappa\geq 0,\my\geq 1, q>n$. Then for each nonnegative radial $u_0\in\Weqom$, \eqref{eq:limprob} has a unique global strong \weqsoln. Furthermore, if $\my>1$ or $\kappa=0$, $u,v$ are bounded in $\Om\times (0,\infty)$.
\end{proposition}

\subsection{Blow-up for small $\my$}
The contrasting -- and more interesting -- case is that of small values of $\my$. Here we will show blow-up. We borrow the following technical tool from \cite{winkler_14_ctexceed}:
\begin{lemma}
\label{thm:prepare_blowup_estimate}
 Let $a>0, b\geq 0, d>0, \kappa>1$ be such that
\[
 a>(\frac{2b}{d})^\frac1\kappa.
\]
Then if for some $T>0$ the function $y\in C([0,T))$ is nonnegative and satisfies 
\[
 y(t)\geq a-bt+d\intnt y^\kappa(s)ds
\]
for all $t\in(0,T)$, we necessarily have 
\[
 T\leq\frac2{(\kappa-1)a^{\kappa-1}d}
\]
\end{lemma}
\begin{proof}
 \cite[Lemma 4.9]{winkler_14_ctexceed}.
\end{proof}

To be of any use to us, this estimate must be accompanied by lower bounds for (some norm of) $u$. We prepare those by the following lemma %(cf. \cite[L. 4.8]{winkler_14_ctexceed}).
\begin{lemma}
\label{thm:ulpgeq}
 Let $\kappa\geq 0$ and $\my>0$. For all $p>1$ and $\eta>0$ there is  $B(\eta,p)>0$ such that for all $q>1$ all \weqsoln s $(u,v)$ of\eqref{eq:limprob} with nonnegative $u_0$ in $\Om\times(0,T)$ satisfy
 \[
  \intom u^p(t)\geq \intom u_0^p+((1-\my)p-1-\eta)\intnt\intom u^{p+1}-B(p,\eta)\intnt\left(\intom u\right)^{p+1}\quad \text{for all } t\in(0,T).
 \]
\end{lemma}

\begin{proof}
 The same testing procedure as in \cite[Lemma 4.8]{winkler_14_ctexceed} leads to success. We repeat it (with the necessary adaptions) for the sake of completeness, because Lemma \ref{thm:ulpgeq} is a main building block of the blow-up result.\\
 Let $T_0\in(0,T)$, $t_0\in(0,T_0)$, $\delta\in(0,T-t_0)$, $\chi_\delta$ as in the proof of Lemma \ref{thm:unique}:
\[
 \chi_\delta(t):=\begin{cases}1&t<t_0,\\
		  \frac{t_0-t+\delta}{\delta}&t\in[t_0,t_0+\delta],\\
                  0&t>t_0+\delta.
                 \end{cases}
\] 
For each $\xi>0$, the function $(u+\xi)^{p-1}$ belongs to $L^\infty_{loc}([0,T),\Weqom)$ and for $\delta\in(0,T_0-t_0), h\in(0,1),\xi>0$,
\[
 \phii(x,t):=\chi_\delta(t)\frac1h\int_{t-h}^t(u(x,s)+\xi)^{p-1}ds,\qquad  (x,t)\in\Om\times(0,T)
\]
is a test function for \eqref{eq:defsoln}, if we set $u(\cdot,t)=u_0$ for $t<0$. This yields
\begin{align*}
\frac1\delta\int_{t_0}^{t_0+\delta}&\intom u(x,t)\frac1h\int_{t-h}^t(u(x,s)+\xi)^{p-1}dsdxdt- \intom u_0(x)(u_0(x)+\xi)^{p-1}dx\\
&\qquad\qquad\qquad\qquad\qquad\qquad\qquad-\intnT\intom \chi_\delta(t)u(x,t)\frac{(u(x,t)+\xi)^{p-1}-(u(x,t-h)+\xi)^{p-1}}hdxdt\\
=&(p-1)\intnT\intom \chi_\delta(t)u(x,t)\na v(x,t) \frac1h\int_{t-h}^t(u(x,t)+\xi)^{p-2}\na u(x,t)dsdxdt\\ 
&\qquad\qquad\qquad\qquad\qquad\qquad+\kappa\intnT\intom \chi_\delta(t)u(x,t)\frac 1h\int_{t-h}^t (u(x,s)+\xi)^{p-1}dsdxdt\\
&\qquad\qquad\qquad\qquad\qquad\qquad\qquad\qquad\qquad -\my\intnT\intom \chi_\delta(t)u^2(x,t)\frac1h\int_{t-h}^t(u(x,s)+\xi)^{p-1}dsdxdt.
\end{align*}
For the sake of brevity, we will omit several instances of $(x,t)$ and $(x,s)$, when there is no danger of confusion. 
Let $\delta$ tend to $0$, use continuity of $u$ and Lebesgue's theorem %($u$ bounded as $q>n$)
to obtain
\begin{align}
\label{eq:456}
 &\intom u(\cdot,t_0)\frac1h\int_{t_0-h}^{t_0}(u(x,s)+\xi)^{p-1}ds
-\intntn\intom u(x,t)\frac{(u(x,t)+\xi)^{p-1}-(u(x,t-h)+\xi)^{p-1}}h\notag\\
&\qquad\qquad\qquad\qquad\qquad\qquad\qquad\qquad\qquad\qquad\qquad\qquad\qquad\qquad\qquad\qquad-\intom u_0(u_0+\xi)^{p-1}\notag\\
&\qquad=(p-1)\intntn\intom u\na v\frac1h\int_{t-h}^t(u+\xi)^{p-2}\na u+\kappa\intntn\intom u\frac1h\int_{t-h}^t(u(x,s)+\xi)^{p-1}\\
&\qquad\qquad\qquad\qquad\qquad\qquad\qquad\qquad\qquad\qquad\qquad\qquad-\my\intntn\intom u^2(x,t)\frac1h\int_{t-h}^t(u+\xi)^{p-1}.\notag
\end{align}
Consider the second integral on the left hand side:
\begin{align*}
 -\intntn\intom u(x,t)&\frac{(u(x,t)+\xi)^{p-1}-(u(x,t-h)+\xi)^{p-1}}h\\
&=-\frac1h\intntn\intom (u(x,t)+\xi)^p+\frac1h\intntn\intom(u(x,t)+\xi)(u(x,t-h)+\xi)^{p-1}\\
&\qquad+\frac\xi h\intntn\intom(u(x,t)+\xi)^{p-1}-\frac\xi h\intntn \intom (u(x,t-h)+\xi)^{p-1}&=I_1+I_2
\end{align*}
where upon an application of Young's inequality, 
\begin{align*}
 I_1&\leq -\frac1h\intntn\intom (u(x,t)+\xi)^p+\frac1{ph}\intntn\intom (u(x,t)+\xi)^p+\frac{p-1}p\frac1h\intntn \intom (u(x,t-h)+\xi)^p\\
&=-\frac{p-1}{ph}\intntn\intom (u(x,t)+\xi)^p+\frac{p-1}{ph}\intntn\intom(u(x,t-h)+\xi)^p\\
&=\frac{p-1}{ph}\left[h\intom (u_0(x)+\xi)^p-\int_{t_0-h}^{t_0}\intom(u(x,t)+\xi)^p\right]
\end{align*}
and by similar cancellations as in the last step
\begin{align*}
 I_2
&=-\xi\intom(u_0(x,t)+\xi)^{p-1}+\frac\xi h \int_{t_0-h}^{t_0}\intom(u(x,t)+\xi)^{p-1}.
\end{align*}
Hence, again by continuity of $u$,
\begin{align*}
 \limsup_{h\to 0} &\left(-\intntn\intom u\frac{(u+\xi)^{p-1}-(u(x,t-h)+\xi)^{p-1}}h\right)\\
&\leq \frac{p-1}p\intom (u_0+\xi)^p-\frac{p-1}p\intom (u(x,t_0)+\xi)^p - \xi\intom(u_0+\xi)^{p-1}+\xi\intom  (u(x,t_0)+\xi)^{p-1}.
\end{align*}

Let $\psi\in L^{q'}(\Om\times(0,T_0),\Rn)$, where $\frac1{q'}+\frac1q=1$.
Then by $u\in L^\infty((-1,T),\Weqom)$, 
\[
 \intom\intnt\frac1h\int_{t-h}^t (u(x,s)+\xi)^{p-2}\na u(x,s)ds\cdot\psi(x,t)dxdt \to \intom\intnt (u(x,t)+\xi)^{p-2}\na u(x,t)\cdot\psi(x,t) dxdt
\]
Therefore, as $h\to 0$, \eqref{eq:456} becomes
\begin{align*}
 \intom u(x,t_0)(u(x,&t_0)+\xi)^{p-1}+\frac{p-1}p\intom(u_0(x)+\xi)^p-\frac{p-1}p\intom(u(x,t_0)+\xi)^p\\
 &\qquad\qquad\qquad\qquad\qquad-\xi\intom(u_0+\xi)^{p-1}+\xi\intom(u(x,t_0)+\xi)^{p-1}-\intom u_0(u_0+\xi)^{p-1}\\
 &\geq ( p-1)\intntn\intom u\na v(u+\xi)^{p-2}\na u+\kappa\intntn\intom u(u+\xi)^{p-1}-\my\intntn\intom u^2(u+\xi)^{p-1}\\
 &\geq (p-1)\intntn\intom u\na v(u+\xi)^{p-2}\na u-\my\intntn\intom u^2(u+\xi)^{p-1}.
\end{align*}
Because $p>1$, $u(u+\xi)^{p-2}\to u^{p-1}$ uniformly as $\xi\to0$. In this limit we therefore obtain
%$u$ is continuous, $p>1$, hence $u(u+\xi)^{p-2}\to u^{p-1}$ uniformly. Letting $\xi\to 0$, we obtain 
\begin{align*}
 \intom u^p(\cdot,t_0)+\frac{p-1}p\intom u_0^p-\frac{p-1}p\intom u^p(\cdot,t_0)-\intom u_0^p
\geq (p-1)\intntn\intom u^{p-1}\na v\na u-\my\intntn\intom u^{p+1}.
\end{align*}
This is equivalent to the following inequality, where we can use the elliptic equation of \eqref{eq:limprob} to express $\Lap v$ differently.
\begin{align*}
 \frac1p\intom u^p(\cdot,t_0)-\frac1p\intom u_0^p&\geq -\frac{p-1}p\intntn\intom u^p\Lap v-\my\intntn\intom u^{p+1}\\
&=\frac{p-1}p\intntn\intom u^{p+1}-\frac{p-1}p\intntn\intom u^pv-\my\intntn\intom u^{p+1}.
\end{align*}
Here Young's inequality and Lemma \ref{thm:vlppeleule} 
 provide constants $C_1$ and $\ctilde$ respectively, such that
\begin{align*}
\frac1p\intom u^p(\cdot,t_0)-\frac1p\intom u_0^p&\geq(\frac{p-1}p-\my)\intntn\intom u^{p+1}-\frac\eta{2p}\intntn\intom u^{p+1}-C_1\intntn\intom v^{p+1}\\
&\geq(\frac{p-1}p-\my-\frac\eta{2p})\intntn\intom u^{p+1}-\frac\eta{2p}\intntn\intom u^{p+1}-\ctilde\intntn\left(\intom u\right)^{p+1}\\
&=\frac1p((1-\my)p-1-\eta)\intntn\intom u^{p+1}-\ctilde\intntn\left(\intom u\right)^{p+1}.\qquad\qquad\qedhere
\end{align*}
\end{proof}

These lemmata can be utilized to decide, which alternative of Theorem \ref{thm:existence} occurs for $\my<1$. It is the same as in case of dimension one (see \cite[Thm. 1.4]{winkler_14_ctexceed}) and can be proven almost identically:
 \begin{thm}
\label{thm:blowup}
  Let $\kappa\geq 0, \my\in(0,1)$. For all $p>\frac1{1-\my}$ there is $C(p)>0$ with the following property:
 Whenever $q>n$ and $u_0\in\Weqom$ is nonnegative, radial and %radial for existence of solution 
 \[
  \norm[\Lpom]{u_0}>C(p)\max\setl{\frac1{|\Om|}\int u_0,\frac \kappa \my},
 \]
 the strong \weqsoln\ of \eqref{eq:limprob} blows up in finite time, i.e. in Theorem \ref{thm:existence}, we have $\Tmax<\infty$ and $\limsuptTmax\norm[\Liom]{u(\cdot,t)}=\infty$.
 \end{thm}
 \begin{proof}
 Let $\eta=\frac{(1-\my)p-1}2>0$, $B(p,\eta)$ as in Lemma \ref{thm:ulpgeq},
\[
 C(p):=\left(\frac{4B(\eta,p)}{(1-\my)p-1}\right)^\frac{1}{p+1}|\Om|^{1+\frac{1}{p(p+1)}}.
\]
Suppose that $\norm[\Lpom]{u_0}>C(p)\max\set{\frac1{|\Om|}\intom u_0,\frac \kappa\my}$ and the corresponding \weqsoln\ of \eqref{eq:limprob} (from Theorem \ref{thm:existence}) is global in time, i.e. $\Tmax=\infty$.

Let $y(t):=\intom u^p(x,t)dx$ for $t\geq 0$. This would define a continuous function on $[0,\infty)$. %($u(\cdot,t)\in\Weqom\eingebettetin\Liom, u\in C(\Ombar\times[0,T))$)

According to Lemma \ref{thm:ulpgeq}, $\kappa\geq 0, \my>0, p>\frac1{1-\my}>1$ and the choice of $B(\eta,p)$ make $y$ satisfy 
\begin{align*}
 y(t)&\geq\intom u_0^p+((1-\my)p-1-\eta)\intnt\intom u^{p+1}ds -B(\eta, p)\intnt\left(\intom u\right)^{p+1}\\
 &\geq y(0)+\frac{(1-\my)p-1}{2}|\Om|^{-\frac1p}\intnt[y(s)]^\frac{p+1}p-B(p,\eta)\intnt\left(\intom u\right)^{p+1} 
\end{align*}
for all $t\geq 0$. 
By Lemma \ref{thm:limproblebd}, for all $t\geq 0$ 
\[
 \intnt\left(\intom u\right)^{p+1}ds\leq \intnt (|\Om|\mhat)^{p+1}ds\leq |\Om|^{p+1}\mhat^{p+1}t, 
\]
where $\mhat=\max\set{\frac1{|\Om|}\intom u_0,\frac\kappa\my}.$
Therefore 
\[
 y(t)\geq y(0)+\frac{(1-\my)p-1}{2}|\Om|^{-\frac1p}\intnt y^{\frac{p+1}p}(s)ds-B(p,\eta)|\Om|^{p+1}\mhat^{p+1}t
\]
for all $t\geq 0$. 
An application of Lemma \ref{thm:prepare_blowup_estimate} with 
$a=y(0)$, $b=B(p,\eta)|\Om|^{p+1}\mhat^{p+1}$, $d=\frac{(1-\my)p-1}2|\Om|^{-\frac1p}$, $\kappa=\frac{p+1}p$
now allows to conclude from 
\[
 a\left(\frac{2b}d\right)^{-\frac1\kappa}=\norm[\Lpom]{u_0}^p\left(\frac{4 B(p,\eta)|\Om|^{p+1}\mhat^{p+1}|\Om|^\frac1p}{(1-\my)p-1}\right)^{-\frac p{p+1}}>1
\]
that -- contradicting our assumption and proving the theorem -- $\Tmax$ must be finite.
 \end{proof}

\section{No thresholds on population density}
\label{sec:nothreshold}

Let us now, finally, prove the main result, corresponding to \cite[Thm. 1.1]{winkler_14_ctexceed} and expanding this to higher dimensional space.

\begin{proof}[Proof of Theorem \ref{thm:main}] (See \cite[Thm. 1.1]{winkler_14_ctexceed}).
 Let $u_0\in\Weqom$ be as in the statement of the theorem and let $T>0$ denote the maximal existence time of the corresponding solution of \eqref{eq:limprob}.
We then know by Theorem \ref{thm:blowup} that $T<\infty$ and 
\begin{equation}
 \label{eq:expl}
 \limsup_{t\upto T}\norm[\Liom]{u(\cdot,t)}=\infty.
\end{equation}
In the following only consider solutions of \eqref{eq:epsprob} that exist at least until time $T$. 
All other solutions blow up earlier according to Lemma \ref{thm:epslocalexistence} and therefore trivially satisfy the theorem.
 
Assume that Theorem \ref{thm:main} were not true. Then there would be $M>0$ and a sequence $\eps_j\to 0$ such that 
\[
 u_{\eps_j}(x,t)\leq M
\]
for all $(x,t)\in\Om\times(0,T)$ and $j\in\N$. 
Therefore, we would obtain convergence by Lemma \ref{thm:kgztosolnifbd}:
\[
 u_{\eps_j}\to\utilde \quad\textrm{in  } C(\Ombar\times[0,T])
\]
and
\[
 v_{\eps_j}\to \vtilde \quad\textrm{ in } C^{2,0}(\Ombar\times[0,T]),
\]
as $j\to \infty$, where $(\utilde,\vtilde)$ is a strong solution of \eqref{eq:limprob}. Because such solutions are unique, $(\utilde,\vtilde)=(u,v)$ and in particular $u=\utilde\leq M$ in $\Om\times(0,T)$, contradicting \eqref{eq:expl}.
\end{proof}

% \bibliographystyle{abbrv}
% \bibliography{lit.bib}

\end{document}